%% file: main.tex
\documentclass[11pt,reqno]{amsart}
 \usepackage[foot]{amsaddr}
\usepackage[a4paper,margin=1in]{geometry}

\usepackage{grffile} 

\usepackage[english]{babel}
\usepackage{amsmath, mathtools}
\usepackage{amsthm}
\usepackage{amsfonts}
\usepackage{amssymb}
\usepackage{graphicx}
\usepackage{epstopdf}
\usepackage{wrapfig}
\usepackage{comment}
\usepackage{mathrsfs}

\usepackage{physics}
\usepackage{amsmath}
\usepackage{tikz}
\usepackage{mathdots}
\usepackage{yhmath}
\usepackage{cancel}
\usepackage{color}
\usepackage{siunitx}
\usepackage{array}
\usepackage{multirow}
\usepackage{amssymb}
\usepackage{gensymb}
\usepackage{tabularx}
\usepackage{extarrows}
\usepackage{booktabs}
\usepackage{subcaption}
\usepackage{algorithm}
\usepackage{algpseudocode}
\usepackage{listings}
\usepackage{listofitems} %
\usepackage[outline]{contour} %
\usepackage{xcolor}

\usepackage{macros}
\usepackage{styles}

\renewcommand\labelenumi{(\roman{enumi})}
\renewcommand\theenumi\labelenumi
\numberwithin{equation}{section}
\mathtoolsset{showonlyrefs=true}

\begin{document}

\title[Galerkin NN-POD for Acoustic and Electromagnetic
Wave Propagation]{
Galerkin Neural Network-POD for Acoustic and Electromagnetic
Wave Propagation in Parametric Domains
}

\author{Philipp Weder$^\dagger$}
\author{Mariella Kast $^\dagger$}
\author{Fernando Henr\'iquez $^\dagger$}
\author{Jan S. Hesthaven $^\dagger$}
\address{$^\dagger$Chair of Computational Mathematics and Simulation Science (MCSS), \'Ecole Polytechnique F\'ed\'erale de Lausanne, Lausanne, Switzerland.}
\email{philipp.weder@epfl.ch}
\email{mariella.kast@epfl.ch}
\email{fernando.henriquez@epfl.ch}
\email{jan.hesthaven@epfl.ch}

\maketitle

\begin{abstract}
We investigate reduced-order models for acoustic and electromagnetic wave problems in parametrically defined domains. The parameter-to-solution maps are approximated
following the so-called Galerkin POD-NN method, which combines the construction of a reduced basis via proper orthogonal decomposition (POD) with neural networks (NNs). 
As opposed to the standard reduced basis method, this approach 
allows for the swift and efficient evaluation of reduced-order solutions
for any given parametric input.

As is customary in the analysis of problems in random or parametrically defined domains,
we start by transporting the formulation to a reference domain.
This yields a parameter-dependent variational problem set on parameter-independent functional spaces. In particular, we consider affine-parametric domain transformations 
characterized by a high-dimensional, possibly countably infinite, parametric input.  
To keep the number of evaluations of the high-fidelity solutions manageable, we propose using low-discrepancy sequences to sample the parameter space efficiently. 
Then, we train an NN
to learn the coefficients in the reduced representation. This approach completely decouples the offline and online stages of the reduced basis paradigm.

Numerical results for the three-dimensional Helmholtz and Maxwell equations confirm the method's accuracy up to a certain barrier and show significant gains in online speed-up compared to the traditional Galerkin POD method.
\end{abstract}

\input{1_Introduction.tex}

\input{2_parametric_pdes.tex}
\input{3_projection_pod.tex}

\input{4_neural_networks.tex}

\input{5_numerical_results.tex}

\section{Concluding Remarks}
\label{sec:concluding_remarks}
In this work, we present a Galerkin POD-NN method for surrogate modeling
of three-dimensional acoustic and electromagnetic  wave problems with
parametric-affine shape deformations. Using readily available results for the Maxwell
cavity problem and our analysis for the Helmholtz impedance problem, we demonstrate the analytic or holomorphic dependency of the 
problem's solution upon such parametric shape deformations.
Based on this analysis, we argue that  computational models for this class of problems are amenable to complexity 
reduction using a projection-based reduced basis scheme irrespective  of the dimensionality
of the parametric domain. 
Following the same argument, the 
map from the parameters to the coefficients of the RB basis representation is also amenable to approximation, done here by
using NNs. Unlike many commonly computational models for which the Galerkin POD-NN has been applied,
the involved quantities are complex-valued. We propose a formulation in which we separate the reduced coefficients
into their real and imaginary parts. In the training stage, we consider each as a separate trainable, real-valued output, thus allowing us to retain an NN with real-valued features. 

Our numerical experiments indicate that the mean field may be a predictor
of reasonable accuracy in some cases, whereas the POD-NN method can
improve this by an order of magnitude or more.
The success of the surrogate model also critically depends on the
hyper-parameter choices of the original problem and interestingly seems
to work well on the Matérn-type decay.
We have further observed that for a given data set, there seems to be an
optimal number of basis functions in terms of efficiency, as no more gains
in accuracy can be achieved by increasing the basis size. 
While our theoretical investigations prove convergence rates of the reduced order approximation independent of the parameter dimension, the cost of accurately creating such a reduced order approximation still scales with the parametric dimension, e.g., via the low-discrepancy series or the number of training points that are required for the NN approximation. In principle, this may be addressed by using tailored, high-dimensional quadrature rules as in 
\cite{longo2021higher}.

In the current work, the considerable computational cost of obtaining high fidelity training data for complex 3D problems, limited the size of our training data set, which also shows in the limited success for the most complex problems.
In future work, it would be interesting to investigate whether adding more training data through a larger snapshot set or via an active learning scheme can overcome the error barriers that we observe. 
Similarly,  including physical knowledge into the NN model, for instance, by adding the residual of the underlying PDE to the loss function, could improve performance when the training data set is of limited size. In addition, the method's performance could be improved by considering a multi-fidelity setup, where cheaply available lower-fidelity training data with e.g., reduced resolution, provides better coverage of the solution manifold.

\appendix
\input{6_appendix_parametric_hol.tex}

\newpage
\input{7_appendix_maxwell_viz.tex}

\clearpage
\bibliography{ref}{}
\bibliographystyle{siam}

\end{document}

%% file: 1_Introduction.tex
\section{Introduction}\label{sec:introduction}
Partial differential equations (PDEs) are a ubiquitous approach to physical modeling in engineering and the applied sciences. However, knowledge of the underlying parameters, such as boundary conditions, source terms, or geometry, is often incomplete. Consequentially,  parameters are often varied during the search for a specific configuration of the system, e.g., in parameter estimation \cite{van2007parameter}, topology optimization \cite{borrvall2003topology}, optimal control \cite{troltzsch2010optimal}, or uncertainty quantification \cite{smith2013uncertainty}.
As these methods require repeated evaluations of the parameterized PDE problem, one speaks of \emph{many-query problems}. %

The standard discretization techniques, such as the finite element, finite difference, or finite volume methods, allow precise approximations of the solutions to parameterized PDEs (pPDEs). Yet, such high-fidelity (HF) approximations are intrinsically linked to a prohibitive computational complexity for many-query or real-time applications. \emph{Reduced-order modeling} (ROM) aims at constructing fast surrogate models to accelerate the computation of the approximate solution to a given pPDE problem while retaining an accuracy comparable to that of HF techniques. One common approach to ROM are \emph{reduced basis} (RB) methods. They are based on an \emph{offline-online paradigm}, where first, a reduced basis is constructed from a number of expensive high-fidelity snapshots (offline step). Then, a fast surrogate model is evaluated in the reduced basis (online step).
For a comprehensive review of the RB method for stationary problems with certified error control, we refer to \cite{hesthaven2016certified,Quarteroni_2016,prud2002reliable,rozza2014fundamentals}.

In this work, we focus on the construction of surrogate models from a reduced basis obtained by
the proper orthogonal decomposition (POD). For the Galerkin POD method, the full-order system is projected onto the reduced space, which leads to a reduced system of equations for the RB coefficients. Although effective, this approach does not yield the desired computational speed-up in the general case, as the full-order solution still needs to be assembled in the online stage to evaluate the PDE operators.
Various intrusive techniques aim at alleviating this issue, for example, by exploiting affine
parameter-dependencies \cite[Chapter 3.3]{hesthaven2016certified} or using hyper-reduction   \cite[Chapter 5]{hesthaven2016certified} to approximate reduced order terms directly.

In contrast, non-intrusive RB methods only rely on HF snapshots and typically construct a surrogate of the map from parameters to RB coefficients. To obtain the response surface, a data-driven regression problem is formulated on the set of RB coefficients of HF snapshots at evaluated parameter location. Then, a query of the surrogate model allows for the solution evaluation of a new, unseen parameter setting. The response surface has to be able to capture the non-linear structure of the underlying parameter-to-solution map of the PDE problem, which makes neural networks \cite{Hesthaven2018} or Gaussian processes \cite{guo2018reduced} suitable choices. %
In this work, we explore the construction of projection-based reduced-order models
for acoustic and electromagnetic scattering, which are modeled by the
Helmholtz and time-harmonic Maxwell equations, respectively. For both cases, we consider computational domains with parametrically defined geometries.
Such models are of interest in a variety of applications, e.g.~metallic meta-materials \cite{Bhattarai_2017},
and the design of thermovoltaic cells \cite{Chen_2007}.

In particular, we are interested in studying affine-parametric 
domain deformations with a high-dimensional, possibly countable infinite, parametric input, which leads to a computationally challenging, high-dimensional parameter-to-solution map.
Numerous techniques exist which are tailored to treat efficiently
problems with high-dimensional inputs, such as sparse grid interpolation and quadrature \cite{ZS17,NTW2008,HHPS2018},
higher-order Quasi-Monte Carlo integration (HoQMC)\cite{DKL14,DLC16}, construction of NN
surrogates \cite{SZ19,HSZ20,OSZ22,HOS22,ABD22}, and model order reduction \cite{Chen_2015}.
As pointed out in \cite{CCS15}, a key property to break the so-called \emph{curse of dimensionality}
in the parameter space is the holomorphic dependence of the parameter-to-solution map.
This property has been established for a variety of problems, including for example 
subsurface flows \cite{CNT2016,HPS2016,HS2022}, time-harmonic electromagnetic wave
scattering \cite{JSZ16,Aylwin2020}, stationary Stokes
and Navier-Stokes equation \cite{CSZ18}, Helmholtz equation
\cite{HSSS15,SW23,GKS21}, and for boundary 
integral operators \cite{HS21,henriquez2021shape,PHJ23,dolz2023parametric,DLM22,henriquez2024reduced}.

\subsection*{Contributions}
We theoretically and computationally study the performance of the POD-NN
applied to  the Helmholtz interior impedance problem and the Maxwell
lossy cavity with affine-parametric shape deformations previously used in \cite{Aylwin2020}.
The parametric holomorphy property, which is crucial for the effectiveness of dimensionality reduction,  has previously been established for the Maxwell lossy cavity problem.
For the sake of completeness, we provide the corresponding proof for the Helmholtz impedance problem. 

To approximate the complex-valued reduced order coefficients,   we modify the  POD-NN approach in \cite{Hesthaven2018} and propose
an NN architecture, which models the real and imaginary parts of each reduced coefficient as separate outputs. We show that this does not affect the holomorphy property, indicating that this modified map can also be emulated by NNs.  To further facilitate the learning task, we propose a centered POD approach. To justify the effectiveness of the POD-NN approach, we further provide a complete convergence analysis for the reduced basis method for the two model problems.

Finally, we present a comprehensive set of numerical experiments in which
we study the effect of parametric domain transformations with different decay structures, wave numbers of the problems, and hyperparameters of the neural network.

\subsection*{Outline}
 In Section \ref{sec:rom_for_parametric_pdes} we 
introduce the notion of pPDEs, together with a thorough description of the Helmholtz impedance
and Maxwell lossy cavity problem in parametric domains.
Section \ref{sec:projection_based_ROM} introduces the
projection-based reduced order modeling for pPDEs and some theoretical properties. 
Next, in Section \ref{sec: neural network POD}, we discuss the Galerkin POD-NN approach and its adaptation to
complex-valued solution spaces in more detail. 
In Section \ref{sec: Numerical results}, we demonstrate the efficacy of our proposed approach and provide insights into the relation between problem complexity and approximation error by varying various hyperparameters.
Lastly, in Section \ref{sec:concluding_remarks}, we draw conclusions about this method and sketch directions of possible future research.

%% file: 2_parametric_pdes.tex
\section{Parametric PDEs: Helmholtz and Maxwell Formulations}
\label{sec:rom_for_parametric_pdes}
In this section, we first state the formulation of pPDEs in a generic fashion, which serves
as a common ground to explain  the reduced
basis method applied to both PDE model problems.
Then, we discuss the weak formulations, well-posedness, and discretizations
in more detail for the 
Helmholtz impedance and Maxwell lossy cavity problems. 

Following previous works on the subject, we introduce the 
Galerkin POD for pPDEs, where we put particular emphasis on parametrically defined domains. 

\subsection{Parametric PDEs}
\label{eq:rom_and_ppdes}
Throughout, let $\text{U} \coloneqq [-1,1]^\IN$ be the parameter space.
Let $V$ be a \emph{complex} Hilbert space 
endowed with the scalar product $(\cdot, \cdot)_V$ and
the induced norm $\norm{\cdot}_V = \sqrt{(\cdot, \cdot)_V}$. 
We denote by $V'$ the anti-dual space of $V$, i.e.
the set of all anti-linear functionals acting on $V$.
Then, the differential (strong) form of a pPDE can be expressed as follows: 
For each $\y\in \text{U}$ we seek $u(\y) \in V$ such that
\begin{equation}\label{eq:generic PPDE problem}
	\mathsf{G}
	\left(
		u(\y);
		\y
	\right) = 0,
	\quad \text{in } V',
\end{equation}
where, for each $\y \in \text{U}$, $\mathsf{G}(\cdot;\y): V \to V'$ is an operator
representing the action of the underlying PDE. 
We express \eqref{eq:generic PPDE problem} in a variational form.
To this end, for each $\y \in \normalfont\text{U}$, we define the parameter 
dependent sesquilinear form $\mathsf{g}(\cdot,\cdot;\y): V \times V  \to \C$,
as
\begin{equation}
	\mathsf{g}(u, v; \y) 
	\coloneqq
	\langle 
		\mathsf{G}
		\left(
			u;\y
		\right), 
		v 
	\rangle_{V' \times V},
	\quad
	\forall u, v \in V,
\end{equation}
being $\langle \cdot, \cdot \rangle_{V'\times V}$ the duality pairing.
Then, the variational formulation of \eqref{eq:generic PPDE problem} reads:
Given $\y \in \text{U}$, find $u(\y) \in V$ such that
\begin{equation}\label{eq: generic variational formulation}
	\mathsf{g}(u(\y), v; \y) 
    	= 0,
	\quad
	\forall v \in V.
\end{equation}
In Sections \ref{sec:helmholtz_problem_parametric} and \ref{eq:sec_maxwell_cavity}, we consider the particular instance of a linear operator $\mathsf{G}$.

\subsection{The Discrete Full-order Model}
\label{sec:discrete_full_order_model}
Let $V_h \subset V$ be a finite dimensional subspace of $V$
of dimension $N_h$ with underlying discretization parameter $h>0$.

The Galerkin approximation of the variational problem stated
in \eqref{eq: generic variational formulation} reads as follows: 
For a given $\y \in \text{U}$, find $u_h(\y) \in V_h$ such that
\begin{equation}\label{eq: generic FE formulation}
	\mathsf{g}(u_h(\y), v_h; \y) 
	= 
	0,
	\quad
	\forall v_h \in V_h.
\end{equation}
Let $\{\varphi_1, \ldots, \varphi_{N_h}\}$ be a basis of $V_h$.
Each $v_h  \in V_h$ admits the following unique
representation
\begin{align}
	v_h
	=
	\sum_{m=1}^{N_h}
	c_m
	\left(
		v_h
	\right)
	\varphi_{m},
\end{align}
where $c_m: V_h \rightarrow \IC$ are linear functionals representing the degrees of freedom
of the FE space $V_h$.
Therefore, each $v_h \in V_h$ can be uniquely represented by
the sequence $\left\{c_m(v_h)\right\}_{m=1}^{N_h}$ according to
\begin{align}\label{eq:rel_dof_FE}
	V_h 
	\ni
	v_h
	\Longleftrightarrow
	{\bf v}_h
	\coloneqq
	\begin{pmatrix}
		c_1\left(v_h\right),& 
		\dots, &
		c_{N_h}\left(v_h\right)
	\end{pmatrix}^\top
	\in 
	\mathbb{C}^{N_h},
\end{align}

Problem \ref{eq: generic FE formulation} can be cast in the 
following algebraic form: For each $\y \in \text{U}$,
find $\mathbf{u}_h(\y) \in \C^{N_h}$ such that
\begin{eqnarray}\label{eq: algebraic FE formulation}
	\mathbf{G}_h(\mathbf{u}_h(\y); \y)
	=
	\boldsymbol{0} \in \C^{N_h},
\end{eqnarray}
where the residual vector $\mathbf{G}_h(\mathbf{u}_h; \y)$ is defined as
\begin{equation}
	\left(
		\mathbf{G}_h(\mathbf{u}_h; \y)
	\right)_i 
	\coloneqq 
	\mathsf{g}(u_h(\y), \varphi_i), 
	\quad
	i = 1, \ldots, N_h.
\end{equation}
The solution vector $\mathbf{u}_h(\y)$ is obtained by
solving a system of linear equations of  size $N_h \times N_h$
for linear problems, while in the non-linear case,
iterative methods must be employed. %

\subsection{Parametric Domain Transformations}
\label{sec: A family of domain mappings}
As previously pointed out, we focus on problem where the parametric input
define shape deformations. 
We consider the following family of parametric
domain transformations with respect to a
bounded, Lipschitz domain $\text{D}_0$ referred to as the 
reference domain: %
For each $\y \in \text{U}$, we define
$\boldsymbol T(\cdot,\y): \text{D}_0 \to \R^3$ as
\begin{equation}\label{eq: sinusidal mapping}
	\boldsymbol{T}(\widehat{\bf x};\y)
	\coloneqq
	\boldsymbol{T}_0(\widehat{\bf x})
	+ 
	\sum_{j \geq 1}
	y_j
	\boldsymbol T_j(\widehat{\bf x}),
	\quad
	\widehat{\bf x}\in \text{D}_0,
	\quad
	\y
	=
	\{y_j\}_{j\geq1}
	\in 
	\text{U},
\end{equation}
where $\boldsymbol{T}_j: \text{D}_0 \rightarrow \mathbb{R}^3$, $j \in \IN_0$.
In principle, we allow for possible countable infinite parametric inputs. However,
in the practical computational implementation we consider a truncation in the parametric dimension.
 
In the following, we work under the assumptions stated below.

\begin{assumption}\label{assumption:parametric_holomorphy}
For $j\in \IN$, 
set $b_j = \norm{\boldsymbol{T}_j}_{W^{1,\infty}\left({\normalfont\text{D}}_0;\mathbb{R}^3\right)}$.
\begin{itemize}
	\item[(i)]
	For each $\y \in \normalfont\text{U}$ the domain transformation
	$\boldsymbol{T}(\cdot;\y)$ is bijective and bi-Lipschitz, and 
	\begin{equation}\label{eq:deformed_domain}
		\normalfont\text{D}(\y)
		\coloneqq \left\{{\bf x} \in \mathbb{R}^3: \, {\bf x} 
		= 
		\boldsymbol{T}(\widehat{\bf x};\y), \; \widehat{\bf x} \in \normalfont\text{D}_0\right\}
	\end{equation}
	defines a bounded, Lipschitz domain in $\mathbb{R}^3$.
	\item[(ii)]
	There exists $p \in (0,1)$ such that
	$\boldsymbol{b} = \{b_j\}_{j\in \mathbb{N}} \in \ell^p(\mathbb{N})$.
\end{itemize}
\end{assumption}

As in \cite{Aylwin2020}, we consider the following
setting for the numerical experiments to be presented ahead
in Section \ref{sec: Numerical results}: $\text{D}_0 \coloneqq (-1,1)^3$, $\boldsymbol{T}_0(\widehat{\bf x}) = \widehat{\bf x}$
together with
\begin{equation}\label{eq:def_perturbation}
    \boldsymbol T_j(\mathbf{x})
    \coloneqq 
    \mu_j
    \left (
    \begin{array}{c}
        0\\
        0\\
        \sin(\pi j \hat{x}_1)
    \end{array}\right), 
    \quad
    \boldsymbol{x} = (\widehat{x}_1,\widehat{x}_2,\widehat{x}_3)^\top
    \in \text{D}_0,
\end{equation}
i.e.~we consider a sinusoidal deformation of the cube in the third coordinate. 
In addition, we use the following two choices for $ \mu_j$.
\begin{itemize}
	\item
	{\bf Algebraic Decay.} We assume that $\mu_j(r,\theta) = \theta j^{-(r+1)}$
	for $\theta>0$ and $r>1$.
    	\item 
    	{\bf Mat\'ern-Like Covariance Decay.} 
	We consider the following sequence mimicking 
	the decay of the eigenvalues of the Mat\'ern-Like Covariance operator, see \cite{williams2006gaussian} (Chapter 4.2.1) for details, i.e. we set
	\begin{equation}
		\mu_j(a,\nu,\theta) 
		= 
		\theta
		\frac{a^\nu}{{(a + \pi^2 j^2)}^{\nu + 0.5}}\frac{\Gamma(\nu +0.5)}{\Gamma(\nu)}, 
		\quad a = \frac{2\nu}{l^2},
	\end{equation}
	where $l$ corresponds to the length scale, 
	$\nu$ models the roughness of the perturbation
	and $\Gamma$ is the Gamma function.
	In particular, choosing smaller length scales $l$ leads to
	reduced importance decay along the first parametric dimensions,
	which in turns leads to a richer structure of perturbations
	compared to the algebraic decay.
	This is visualized ahead in the numerical results section.
	
\end{itemize}

\begin{remark}
Observe that for the first case, i.e. the algebraic decay, one has 
$\left\{\mu_j(r,\theta)  \right\}_{j\in \mathbb{N}}  \in \ell^p(\mathbb{N})$ for $ \frac{1}{r}<p<1$,
and in the second case $\left\{\mu_j(a,\nu,\theta)  \right\}_{j\in \mathbb{N}}  \in \ell^p(\mathbb{N})$
for any $\frac{1}{2\nu}<p<1$.
\end{remark}

\subsection{The Helmholtz Impedance Problem in Parametric Domains}
\label{sec:helmholtz_problem_parametric}
Herein, we consider the Helmholtz problem equipped with impedance boundary
conditions on a parametrically defined, bounded Lipschitz domain
$\text{D}(\y)$, $\y \in \text{U}$, with boundary $\Gamma(\y) \coloneqq \partial \text{D}(\y)$.
Here, $\text{D}(\y)$ is as in \eqref{eq:deformed_domain},~i.e. for each $\y \in \text{U}$
the image through the affine-parametric domain transformation of the reference domain $\text{D}_0$.

Let $f \in L^2(\text{D}(\y)), \,g \in L^2(\Gamma(\y))$ and $\kappa > 0$ be given.
For each $\y \in \text{U}$, we consider the problem of finding $u: \text{D}(\y) \rightarrow \C$
such that
\begin{align}\label{eq: helmholtz problem with impedance boundary condition}
	-\Delta u(\y) - \kappa^2 u(\y) 
	= 
	f \quad \text{in }
	\text{D}(\y)
	\quad
	\text{and}
	\quad
	\frac{
		\partial u
	}{
		\partial{\boldsymbol{\nu}_{\Gamma(\y)}}
	} 
        - 
        \imath 
        \kappa 
        \gamma_{\Gamma(\y)} 
        u(\y) 
        = g
        \quad 
        \text{on } \Gamma(\y),
\end{align}
where
\begin{equation}\label{eq:trace_operators}
	\frac{
		\partial u
	}{
		\partial{\boldsymbol{\nu}_{\Gamma(\y)}}
	}
	:
	H^1(\text{D}(\y),\Delta)
	\rightarrow
	H^{-\frac{1}{2}}(\Gamma(\y))
	\quad
	\text{and}
	\quad
	\gamma_{\Gamma(\y)} 
	:
	H^1(\text{D}(\y))
	\rightarrow 
	H^{\frac{1}{2}}(\Gamma(\y))
\end{equation}
denote the Neumann and Dirichlet trace operators.
In \eqref{eq:trace_operators}, we have that
$
	H^1(\text{D}(\y),\Delta) 
	\coloneqq \left \{  u \in H^1(\text{D}(\y)): \, \Delta u \in L^2(\text{D}(\y))\right\},
$
and $\boldsymbol{\nu}(\y) $ signifies the outward-pointing normal vector to $\Gamma(\y)$.
The interior impedance Helmholtz problem admits the following variational formulation.

\begin{problem}[Helmholtz Impedance Problem in $\normalfont\text{D}(\y)$]\label{pbrm:helmholtz}
Let $f \in L^2(\normalfont\text{D}(\y)), \,g \in L^2(\Gamma(\y))$ for each $\y \in \normalfont\text{U}$.
For each $\y \in \normalfont\text{U}$ we seek
$u(\y) \in H^1(\normalfont\text{D}(\y))$ such that
\begin{equation}	
	\mathsf{a}(u(\y), v;\y) 
	= 
	\ell(v;\y),
	\quad
	\forall v \in H^1(\normalfont\text{D}(\y)),
\end{equation}
where the parameter-dependent sesquilinear form
$\mathsf{a}(\cdot,\cdot;\y): H^1(\normalfont\text{D}(\y)) \times H^1(\normalfont\text{D}(\y)) \rightarrow \IC$
is defined for each $\y\in \normalfont\text{U}$ as
\begin{align}
\label{eq: Helmholtz sesquilinear form physical}
	\mathsf{a}
	\left(
		u, v;\y
	\right) 
	\coloneqq
	\int\limits_{\normalfont\text{D}(\y)} 
	\left(
		\nabla u \cdot \nabla \overline{v} 
		- 
		\kappa^2 u \overline{v}
	\right) 
	\normalfont\text{d}{\bf x} 
	- 
	\imath 
	\kappa 
	\int\limits_{\Gamma(\y)} 
	\gamma_{\Gamma(\y)} u   \gamma_{\Gamma(\y)}  \overline{v} 
	\,
	\text{ds},
	\quad
	\forall
	u,v \in H^1(\normalfont\text{D}(\y))
\end{align}
and the parameter-dependent anti-linear form
$\ell(\cdot;\y): H^1(\normalfont\text{D}(\y)) \rightarrow \mathbb{C}$
is defined for each $\y \in \normalfont\text{U}$ as
\begin{align}\label{eq: Helmholtz antilinear form physical}
	\ell(v;\y)
	\coloneqq
	\int\limits_{\normalfont\text{D}(\y)} 
	f \overline{v} 
	\,
	\normalfont\text{d}{\bf x} 
	+ 
	\int\limits_{\normalfont\Gamma(\y)}
	g   \gamma_{\Gamma(\y)}  \overline{v}
	\,
	\text{ds}_{\bf x},
	\quad
	\forall
	v \in H^1(\normalfont\text{D}(\y)).
\end{align}
\end{problem}

By invoking the Banach-Nečas-Babuška theorem, Gårding's inequality, and the injectivity
of the sesquilinear form $\mathsf{a}(\cdot,\cdot;\y): H^1(\normalfont\text{D}(\y))\times H^1(\normalfont\text{D}(\y)) \rightarrow \mathbb{C}$,
one may establish well-posedness of Problem \ref{pbrm:helmholtz}
(cf. \cite[Theorem 35.5]{Ern_2021_II}), pointwise for each $\y \in \normalfont\text{U}$.

\subsubsection{Helmholtz problem in the reference domain}

To recast the integrals in \eqref{eq: Helmholtz sesquilinear form physical}
and \eqref{eq: Helmholtz antilinear form physical} in terms of $\text{D}_0$,
we recall the following formulas for domain and boundary transformations,
which can be found in \cite{Boffi_2013}.
Denote by $d\boldsymbol T(\y)$ the Jacobian matrix of the transformation
$\boldsymbol T(\y)$ and by $J(\y)$ its Jacobian determinant. 
For each $\y\in \text{U}$ we define $\Phi(\y) : H^1(\text{D}(\y)) \rightarrow H^1(\text{D}_0)$
as $\left(\Phi(\y) v\right)(\widehat{\bf x}) = v(\boldsymbol{T}(\widehat{\bf x};\y))$, $\widehat{\bf x} \in \text{D}_0$,
which is usually referred to as the \emph{plain pullback} operator.
One has that for each $\y \in \text{U}$, $\Phi(\y)$ is a bounded linear operator with a bounded inverse, 
see e.g. \cite[Section 3]{CSZ18} and \cite[Lemma 1]{HPS2016}.

Furthermore, for each $\y \in \text{U}$ we define the surface Jacobian on
the boundary $\Gamma(\y)$ as
\begin{equation}
	J_{\text{S}}(\y) 
	= 
	J(\y) 
	\norm{
		d\boldsymbol T^{-\top}(\y)
		\widehat{\boldsymbol{\nu}}(\y)
	}
	\in L^\infty(\text{D}_0)
\end{equation}
where $\widehat{\boldsymbol{\nu}}(\y) \coloneqq \Phi(\y)  {\boldsymbol{\nu}}_{\Gamma(\y)} \in L^\infty(\Gamma_0;\IR^3)$.
For $\widehat{v} \in H^1(\text{D}_0)$ we set
$v \coloneqq \Phi^{-1}(\y) \widehat{v} $.
Then, we have
\begin{equation}
	\nabla v 
	= 
	d\boldsymbol{T}^{-\top}\widehat{\nabla} \widehat{v} \circ \boldsymbol T^{-1},
	\quad
	\int\limits_{\text{D}(\y)} 
	v 
	\,
	\normalfont\text{d}{\bf x}  
	=
	\int\limits_{\text{D}_0} 
	\widehat{v} 
	J(\y) 
	\normalfont\text{d}\widehat{\bf x},
	\quad
	\text{and}
	\quad
	\int\limits_{\Gamma(\y)} 
	v 
	\text{ds}
	= 
	\int\limits_{\Gamma_0} 
	\widehat{v} 
	J_{\text{S}}(\y)
	\text{d}\widehat{\text{s}}_{\widehat{\bf x}},
\end{equation}
where the latter identity is also called \emph{Nanson's formula},
and $\widehat{\nabla}$ denotes the grandient operator in the reference domain.
This allows us to state the variational formulation
for the Helmholtz equation in the reference domain.

\begin{problem}[Helmholtz Impedance Problem in the Reference Domain]\label{eq: Helmholtz nominal problem}
For each $\y \in \normalfont\text{U}$, we seek $\widehat{u}(\y) \in H^1({\normalfont\text{D}}_0)$ such that
\begin{eqnarray}%
	\widehat{\mathsf{a}}
	\left(
		\widehat{u}(\y)
		, 
		\widehat{v}
		;
		\y
	\right)
	= 
	\widehat{\ell}
	\left(
		\widehat{v};\y
	\right)
	,
	\quad
	\forall \widehat{v} \in H^1({\normalfont\text{D}}_0),
\end{eqnarray}
where for each $\y \in \normalfont\text{U}$ the sesquilinear form
$\widehat{\mathsf{a}}(\cdot,\cdot;\y): H^1({\normalfont\text{D}}_0) \times H^1({\normalfont\text{D}}_0) \rightarrow \IC$
is defined as
\begin{equation}\label{eq:sesquilinear_Helmholtz}
\begin{aligned}
	\widehat{\mathsf{a}}
	\left(
		\widehat{v}
		, 
		\widehat{w}
		;
		\y
	\right) 
	\coloneqq
	&
	\int\limits_{{\normalfont\text{D}}_0} 
	\left(
		d\boldsymbol{T}^{-\top}(\y) \widehat{\nabla} \widehat{v} 
		\cdot  
		d\boldsymbol{T}^{-\top}(\y) \widehat{\nabla} \overline{\widehat{w}} 
		- 
		\kappa^2 \widehat{v} \overline{\widehat{w}}
	\right) 
	J(\y)
	\normalfont\text{d}\widehat{\bf x}
	\\
	&
	- 
	\imath \kappa 
	\int\limits_{\Gamma_0} 
	\widehat{v} 
	\overline{\widehat{w}} 
	J_{\text{S}}(\y) 
	\normalfont\text{d}{\text{s}}_{\widehat{\bf x}},
	\quad
	\forall
	\widehat{v},
	\widehat{w}
	\in H^1({\normalfont\text{D}}_0),
\end{aligned}
\end{equation}
and the anti-linear form $\widehat{\ell}(\cdot;\y): H^1({\normalfont\text{D}}_0)\rightarrow \IC$  
is defined as
\begin{align}\label{eq:anti_linear_Helmholtz}
	\widehat{\ell}(\widehat{v};\y)
	\coloneqq
	\int\limits_{{\normalfont\text{D}}_0} 
	\widehat{f}(\y) \overline{\widehat{v}} 
	J(\y) 
	\normalfont\text{d}\widehat{\bf x}
	+ 
	\int\limits_{\Gamma_0} 
	\widehat{g}(\y) 
	\gamma_{\Gamma_0}\overline{\widehat{v}} 
	J_{\text{S}}(\y) 
	\text{d}{\text{s}}_{\widehat{\bf x}},
	\quad
	\forall
	v \in H^1({\normalfont\text{D}}_0),
\end{align}
where $\widehat{f}(\y) \coloneqq \Phi(\y)f \in L^2(\normalfont\text{D}_0)$ and $\widehat{g}(\y) \coloneqq \Phi(\y)g \in L^2(\Gamma_0)$.
\end{problem}

One can readily see that the unique solvability of Problem \ref{pbrm:helmholtz} together
with the properties of the plain pullback operator straightforwardly entail the well-posedness of
Problem \ref{eq: Helmholtz nominal problem}, and indeed it holds $\Phi(\y)u(\y) = \widehat{u}(\y)$ for each 
$\y \in \text{U}$, where $u(\y) \in H^1(\text{D}(\y))$ and $\widehat{u}(\y)\in H^1(\text{D}_0)$ are the unique solutions 
to Problems \ref{pbrm:helmholtz} and \ref{eq: Helmholtz nominal problem}, respetively.

\subsubsection{Discrete Full-Order Model for Helmholtz Impedance Problem}
\label{sec:discrete_FOM_Helmholtz}
The well-posedness of the discrete problem with $H^1$-conforming finite elements carries over
from the case without domain parametrization. Hence, we will later use continuous Lagrangian
FE defined on a suitable mesh of $\text{D}_0$ to approximate Problem \ref{eq: Helmholtz nominal problem}.
Consequently, the model presented in this section fits the general framework introduced in 
Section \ref{eq:rom_and_ppdes} and \ref{sec:discrete_full_order_model}: the role of the Hilbert space $V$ in Section \ref{eq:rom_and_ppdes}
is played by the Sobolev space $H^1(\text{D}_0)$, the role of $V_h$ by the continuous Lagrangian
FE, and $\mathsf{g}(\cdot,\cdot;\y)=\widehat{\mathsf{a}}(\cdot,\cdot;\y) - \widehat{\ell}(\cdot;\y)$ for each $\y \in \text{U}$, with $\widehat{\mathsf{a}}\left(\cdot, \cdot;\y\right)$ as in \eqref{eq:sesquilinear_Helmholtz} and $\widehat{\ell}(\cdot;\y)$ as in \eqref{eq:anti_linear_Helmholtz}.

Regarding the convergence of the full-order model for the Helmholtz impedance 
problem with respect to the discretization of the FE space, we have the following result: 
For each $\y \in \text{U}$ there exists $h_0(\y)$ (i.e.~depending on $\y$) such that for
$h<h_0(\y)$ it holds
\begin{equation}
	\norm{
		\widehat{u}(\y)
		-
		\widehat{u}_h(\y)
	}_{H^1(\text{D}_0)}
	\leq
	C(\y)
	\inf_{v_h \in V_h}
	\norm{
		\widehat{u}(\y)
		-
		v_h
	}_{H^1(\text{D}_0)},
\end{equation}
for some $C(\y)>0$, where $\widehat{u}_h(\y)$ is solution to the following variational problem
\begin{eqnarray}%
	\widehat{\mathsf{a}}
	\left(
		\widehat{u}_h(\y)
		, 
		\widehat{v}_h
		;
		\y
	\right)
	= 
	\widehat{\ell}
	\left(
		\widehat{v}_h;\y
	\right)
	,
	\quad
	\forall \widehat{v}_h \in V_h
\end{eqnarray}
As it is customary for variational problems satisfying Garding-type inequalities, a minimal level 
of resolution of the FE space is required to obtain quasi-optimality, and, therefore, convergence of the Galerkin method. It is important to point out that the minimal level of resolution $h_0(\y)$ and
$C(\y)$, which in principle depend of the particular instance of $\y \in \text{U}$, can be made independent of $\y \in \text{U}$ by using a finite covering argument and the compactness
of $\text{U}$.

\subsection{The Maxwell Lossy Cavity Problem in Parametric Domains}
\label{eq:sec_maxwell_cavity}
As a second model problem, we consider a time-harmonic electromagnetic
cavity problem with circular frequency $\omega > 0$ in a parametrically defined
bounded, Lipschitz domain $\text{D}(\y)$, for each $ \y \in \text{U}$ as in \eqref{eq:deformed_domain}.

For simplicity, we consider a constant complex domain conductivity $\sigma$
as well as a constant complex dielectric permittivity $\varepsilon$ and magnetic
permeability $\mu$.
In addition, we also introduce the quantity
$\Lambda \coloneqq \omega^2 \varepsilon - \imath \omega \sigma$.
As in \cite{Aylwin2020}, we make the following assumptions:
There exists $\vartheta \in [0, 2\pi)$ such that
\begin{equation}\label{eq:assumption_on_maxwell_constants}
	\mu_b 
	\coloneqq
	\Re
	\left\{
		e^{\imath \vartheta} \mu^{-1}
	\right\} 
	> 0 
	\quad
	\text{and}
	\quad
	\Lambda_b
	\coloneqq
	\Re
	\left\{
		-e^{\imath\vartheta} \Lambda
	\right\} 
	> 0.
\end{equation}
Lastly, we assume a continuous source current density,
i.e.~for each $\y\in \text{U}$ we assume
$\boldsymbol J \in \mathscr{C}^0\left(\text{D}(\y);\C^3\right)$.

In this simplified setting, let $\boldsymbol E$ and $\boldsymbol H$ be the
complex-valued electric and magnetic fields, respectively. Maxwell equations in $\text{D}(\y)$
read (see, e.g.,~\cite{Monk})
\begin{align}
\begin{cases}
	\curl\boldsymbol E + \imath \omega \mu \boldsymbol H 
	&= 0
	\quad
	\text{in } \text{D}(\y),\\
    	(\imath \omega \varepsilon + \sigma) \boldsymbol E - \curl \boldsymbol H 
	&= 
	-\imath\omega \boldsymbol J
	\quad
	\text{ in } \text{D}(\y).
\end{cases}
\end{align}
By defining the quantity $\kappa^2 \coloneqq \omega^2 \mu \varepsilon - \imath \omega \mu \sigma$
and applying the curl operator to the first equation, we can reduce the system to
\begin{equation}\label{eq: maxwell reduced}
	\curl (\curl \boldsymbol E) 
	- 
	\kappa^2 \boldsymbol E 
	= 
	-
	\imath \omega \mu 
	\boldsymbol J
	\quad
	\text{in }
	\text{D}(\y)
\end{equation}
On the boundary, we assume perfect electrical conductor (PEC) boundary conditions
\begin{equation}
	\gamma_d^\times \boldsymbol E 
    	= 
	\boldsymbol{0}
	\quad
	\text{on }
	\Gamma(\y),
\end{equation}
where $\gamma_d^\times$ denotes the flipped tangential trace, i.e.
$\gamma_d^\times \boldsymbol E \coloneqq \gamma_d\left(\boldsymbol n \times \boldsymbol E\right)$. Thus, only the tangential component electric field on the boundary of the domain vanishes.\\

We define $H(\Curl; \text{D}(\y))$ and $H_0(\Curl; \text{D}(\y))$ as
\begin{align}
    H(\Curl; \text{D}(\y)) &\coloneqq \{\boldsymbol u \in L^2(\text{D}(\y))^3 \mid \curl \boldsymbol u \in L^2(\text{D}(\y))^3\},\\
    H_0(\Curl; \text{D}(\y)) &\coloneqq \{\boldsymbol u \in H(\Curl; \text{D}(\y)) \mid \gamma^\times_D \boldsymbol u = 0\},
\end{align}
and we equip them with the norm
\begin{equation*}
	\norm{
		\boldsymbol u
	}_{H\left(\Curl; \text{D}(\y)\right)} 
	\coloneqq
	\norm{
		\boldsymbol u
	}_{L^2\left(\text{D}(\y)\right)} 
	+ 
	\norm{
		\curl \boldsymbol u
	}_{L^2\left(\text{D}(\y)\right)}.
\end{equation*}

The variational formulation of \eqref{eq: maxwell reduced} on the physical domain $\text{D}(\y)$
reads as follows.

\begin{problem}[Maxwell Cavity Problem in $\normalfont\text{D}(\y)$]\label{prb:maxwell_variational}
For each $\y \in \normalfont\text{U}$, we seek
$\boldsymbol{E}(\y) \in H_0\left(\Curl;\normalfont\text{D}(\y)\right)$
such that
\begin{equation}\label{eq:variational_maxwell_problem}
        \mathsf{a}
        \left(
        		\bE(\y)
		,
		\bV; \y
	\right)
        = 
        \ell
        \left(
        		\boldsymbol{V};\y
	\right),
        \quad \forall 
        H_0\left(\Curl; \normalfont\text{D}(\y)\right)
\end{equation}
with the parameter-dependent sesquilinear form
$\mathsf{a}(\cdot,\cdot;\y): H_0\left(\Curl;\normalfont\text{D}(\y)\right) \times H_0\left(\Curl;\normalfont\text{D}(\y)\right)\rightarrow \IC $ form
\begin{align}
	\mathsf{a}
	\left(
		\bV, \bW;\y
	\right)
	\coloneqq
	\int\limits_{\normalfont\text{D}(\y)} 
	\left(
		\mu^{-1} 
		(\curl \bV) \cdot (\curl \overline{\bW})  
		- 
		\Lambda 
		\bV 
		\cdot
		\overline{\bW} 
	\right)
	\normalfont\text{d}{\bf x},
	\quad
	\forall
	\bV, \bW
	\in 
	H_0\left(\Curl;\normalfont\text{D}(\y)\right)
\end{align}
and the parameter-dependent anti-linear form
$\ell(\cdot;\y): H_0\left(\Curl;\normalfont\text{D}(\y)\right)  \rightarrow \IC $
\begin{align}
	\ell
	\left(
		\bV;\y
	\right)
	\coloneqq 
	- \imath
	\omega 
	\int\limits_{\normalfont\text{D}(\y)}
	\boldsymbol J \cdot \overline{\bV}
	\normalfont\text{d}{\bf x},
	\quad
	\bV \in H_0\left(\Curl;\normalfont\text{D}(\y)\right),
\end{align}
where $\boldsymbol J \in \mathscr{C}^0\left(\normalfont\text{D}(\y);\C^3\right)$.
\end{problem}

Due to the assumptions in \eqref{eq:assumption_on_maxwell_constants},
we can apply the Lax-Milgram lemma to show the well-posedness of
Problem \ref{prb:maxwell_variational}.
Indeed, for a fixed parameter $\y \in \text{U}$ and any
$\bV \in H_0\left(\Curl;\normalfont\text{D}(\y)\right)$ it holds
\begin{equation}\label{eq:proof_coercivity_Maxwell}
\begin{aligned}
	\snorm{\mathsf{a}(\bV, \bV;\y)}
	\geq 
	\min \{\mu_b, \Lambda_b\} 
	\norm{
		\bV
	}^2_{H\left(\Curl;\normalfont\text{D}(\y)\right)},
\end{aligned}
\end{equation}
with $\mu_b, \Lambda_b$ as in \eqref{eq:assumption_on_maxwell_constants}.

Consequently, for each $\y \in \text{U}$ the sesquilinear $\mathsf{a}(\cdot,\cdot;\y)$ is
coercive according to \eqref{eq:proof_coercivity_Maxwell}.
Observe that the coercivity constant does not depend on $\y \in \normalfont\text{U}$.
In addition, since both $\mathsf{a}(\cdot,\cdot;\y)$ and $\ell(\cdot;\y)$ are continuous
the Lax-Milgram lemma applies, thus yielding well-posedness of
Problem \ref{prb:maxwell_variational} for each $\y \in \text{U}$.

\subsubsection{Maxwell Cavity Problem in the Reference Domain}
One noteworthy difference tothe Helmholtz problem is
that we have to define the pullback operator $\Phi: V \to V_0$ differently.
In fact, to preserve $H(\Curl; \text{D})$-confomity of the fields, we need to use the following
domain transformation, defined for each $\y \in \normalfont\text{U}$ as
\begin{eqnarray}\label{eq: Maxwell pullback operator}
	\Phi(\y)(\boldsymbol v)
	\coloneqq
	d\boldsymbol{T}(\y)^{-\top}
	(\boldsymbol v \circ \boldsymbol T(\y)),
	\quad
	\forall \boldsymbol v \in V(\y).
\end{eqnarray}

As discussed in \cite[Lemma 2.2]{Ern_2017}, for each $\y \in \normalfont\text{U}$ the pullback operator
$\Phi(\y)$ in \eqref{eq: Maxwell pullback operator} admits 
bounded extension $\Phi(\y) \in \mathscr{L}_{\normalfont\text{iso}}\left(\hcurl, H(\Curl; \text{D}(\y))\right)$.
In addition, for each $\boldsymbol u \in H(\Curl; \text{D}(\y))$
it holds that
\begin{equation}
	\curl \left(\Phi(\y) \boldsymbol u\right)
	= 
	\det(d\boldsymbol{T}(\y)) d\boldsymbol{T}^{-1}(\y)((\curl \boldsymbol u) \circ \boldsymbol{T}(\y)).
\end{equation}

For ease of notation, let us also define the pullback of the current density
$\boldsymbol{J}(\y) \coloneqq \boldsymbol{J} \circ \boldsymbol{T}_y \in \mathscr{C}^0\left(\normalfont\text{D}_0;\C^3\right)$.
With the previous Lemma at hand, we can now state the nominal variational
problem for the Maxwell cavity problem (cf.~\cite{Aylwin2020}).

\begin{problem}[Maxwell Cavity Problem in the Reference Domain]
For each $\y \in \normalfont\text{U}$, we seek
$\widehat{\boldsymbol{E}}(\y) \in H_0\left(\Curl;\normalfont\text{D}_0\right)$
such that
\begin{equation}\label{eq: Maxwell nominal problem}
        \widehat{\mathsf{a}}
        \left(
       	\widehat{\boldsymbol{E}}(\y)
        , 
      	\widehat{\boldsymbol{V}};
	\y
       	\right)
        = 
        \widehat{\ell}
        \left(\widehat{\boldsymbol{V}};\y\right), 
        \quad \forall \widehat{\boldsymbol{v}} \in H_0\left(\Curl;\normalfont\text{D}_0\right),
\end{equation}
where the sesquilinear form
$\widehat{\mathsf{a}}(\cdot,\cdot;\y): H_0\left(\Curl;\normalfont\text{D}_0\right) \times H_0\left(\Curl;\normalfont\text{D}_0\right)\rightarrow \mathbb{C}$
is given by
\begin{equation}\label{eq:sesquilinear_Maxwell}
\begin{aligned}
	\widehat{\mathsf{a}}\left(\widehat{\boldsymbol{u}}, \widehat{\boldsymbol{v}};\y\right) 
	= 
	&
	\int\limits_{\normalfont\text{D}_0} 
	J(\y)^{-1} 
	\mu^{-1}
	\left(
		d\boldsymbol{T}(\y) (\curl \widehat{\boldsymbol{u}})
	\right) 
	\cdot
	\left(
		d\boldsymbol{T}(\y)(\curl \overline{\widehat{\boldsymbol{v}}})
	\right) 
	\\
	&
	- 
	\Lambda J(\y) 
	d\boldsymbol{T}^{-\top}(\y)
	\widehat{\boldsymbol{u}} 
	\cdot 
	d\boldsymbol{T}^{-\top}(\y)
	\overline{\widehat{\boldsymbol{v}}} 
	\normalfont\text{d}\widehat{\boldsymbol{x}}
\end{aligned}
\end{equation}
and, for each $\y \in \normalfont\text{U}$,
the anti-linear form $\widehat{\ell}(\y): H_0\left(\Curl;\normalfont\text{D}_0\right) \rightarrow \mathbb{C}$
is given by
\begin{align}\label{eq:anti_linear_Maxwell}
	\widehat{\ell}(\widehat{\boldsymbol v};\y)
	\coloneqq
	-\imath \omega 
	\int\limits_{\normalfont\text{D}_0} 
	J(\y) 
	\boldsymbol{J}(\y) 
	\cdot d\boldsymbol T^{-\top}(\y) \overline{\widehat{\boldsymbol{v}}}
	\normalfont\text{d}\widehat{\boldsymbol{x}}.
\end{align}
\end{problem}

\subsubsection{Discrete Full-Order Model for Maxwell Cavity Problem}
\label{sec:discrete_FOM_Maxwell}
The Maxwell cavity problem is approximated in an $H(\Curl; \cdot)$-conforming 
fashion using Nédélec elements as described in the reference works \cite{Boffi_2013, Ern_2021_I},
and \cite{Monk}.

The model presented in this section fits the general framework introduced in 
Section \ref{eq:rom_and_ppdes} and \ref{sec:discrete_full_order_model}:  the role of the Hilbert space $V$ in Section \ref{eq:rom_and_ppdes}
is played by the space $H_0(\Curl;\text{D}_0)$, the role of $V_h$ by the curl-conforming Nédélec elements
as discussed previously, whereas $\mathsf{g}$ is replaced by
$\mathsf{g}(\cdot,\cdot;\y)=\widehat{\mathsf{a}}(\cdot,\cdot;\y) - \widehat{\ell}(\cdot;\y)$ for each $\y \in \text{U}$, with $\widehat{\mathsf{a}}\left(\cdot, \cdot;\y\right)$ as in 
\eqref{eq:sesquilinear_Maxwell} and $\widehat{\ell}(\cdot;\y)$ as in \eqref{eq:anti_linear_Maxwell}.

Indeed, as a consequence of the assumption stated
in \eqref{eq:assumption_on_maxwell_constants}  Cea's lemma holds as
\begin{equation}
	\norm{
		\widehat{\boldsymbol{E}}(\y)
		-
		\widehat{\boldsymbol{E}}_h(\y)
	}_{H^1(\text{D}_0)}
	\leq
	C
	\inf_{v_h \in V_h}
	\norm{
		\widehat{\boldsymbol{E}}(\y)
		-
		v_h
	}_{H^1(\text{D}_0)},
\end{equation}
where $\widehat{\boldsymbol{E}}_h(\y)  \in V_h $ is the unique solution to 
the following variational problem
\begin{equation}%
        \widehat{\mathsf{a}}
        \left(
       	\widehat{\boldsymbol{E}}_h(\y)
        , 
      	\widehat{\boldsymbol{V}}_h;
	\y
       	\right)
        = 
        \widehat{\ell}
        \left(\widehat{\boldsymbol{V}}_h;\y\right), 
        \quad \forall \widehat{\boldsymbol{V}}_h \in V_h.
\end{equation}

%% file: 3_projection_pod.tex
\section{Projection-based Reduced Order Modeling}
\label{sec:projection_based_ROM}
The Galerkin approximation of the shape-parametric Helmholtz impedance
and Maxwell lossy cavity problems,
as presented in Sections \ref{sec:helmholtz_problem_parametric} and \ref{eq:sec_maxwell_cavity},
for each parametric input entails a high computational cost. 
For many-query applications or real-time computations, one needs a fast and accurate methodology to evaluate the parameter-to-solution map for each particular instance of the parametric input. 

This motivates the use of model order reduction techniques such as
the reduced basis method. Instead of seeking a solution
for each parametric input in a suitable finite-dimensional subspace, we solve the problem 
in a reduced space of a much smaller dimension than that of the full-order model. 

\subsection{The Reduced Basis Method}\label{sec:reduced_basis_method}
Assume that we have access to a reduced basis $V^{(\text{rb})}_L$ of
dimension $L \ll N_h$ of the form
$V^{(\text{rb})}_L  = \text{span}\{\psi_1, \ldots, \psi_L\}\subset V_h$.
We discuss one possible way to construct such a basis in Section \ref{sec:POD}.
In practical applications one considers only finitely many parametric inputs. 
To this end, we define $\truncparamspace\coloneqq [-1,1]^{J}$, where $J \in \mathbb{N}$
corresponds to the parametric dimension.
Then, for a given parametric input $\y \in \truncparamspace$,
we seek solutions $u^{\text{(rb)}}_L(\y) \in V^{(\text{rb})}_L $ of the form
\begin{align}
\label{eq: generic reduced basis function}
	u^{\text{(rb)}}_L(\y)
	= 
	\sum_{\ell = 1}^L
	\left(
		\mathbf{u}^{\text{(rb)}}_L(\y)
	\right)_\ell
	\psi_\ell \in V^{\text{(rb)}}_L,
\end{align}
with $\mathbf{u}^{\text{(rb)}}(\y) \in \C^{L}$ being the \emph{reduced coefficients}.

In order to compute the reduced basis solution $u^{\text{(rb)}}_L(\y)$
for a particular $\y \in \truncparamspace$, we follow a standard Galerkin approach,
see e.g.~\cite{Quarteroni_2016}.
We project the variational problem \eqref{eq: generic variational formulation}
onto the reduced space $V^{\text{(rb)}}_L$, thus yielding the following 
\emph{reduced basis problem}: Given $\y \in \truncparamspace$, find $u^{\text{(rb)}}_L(\y) \in V^{\text{(rb)}}_L$ such that
\begin{equation}\label{eq:generic_rb_var_form}
	\mathsf{g}
	\left(
		u^{\text{(rb)}}_L(\y)
		,
		v^{\text{(rb)}}_L
		; \y
	\right) 
	= 
	0,
	\quad \forall v^{\text{(rb)}}_L \in V^{\text{(rb)}}_L.
\end{equation}
As in Section \ref{sec:discrete_full_order_model},
each $\psi_\ell \in V^{(\text{rb})}_L $ can be uniquely represented by
the sequence $\{c_m(\psi_\ell)\}_{m=1}^{N_h}$, which we gather in
the vector $\boldsymbol{\psi}_\ell$.
Next, we define the \emph{reduced basis matrix} as
\begin{equation}
	\mathbb{V}^{\text{(rb)}}_L
	\coloneqq
	\left(
		\boldsymbol \psi_1, \ldots , \boldsymbol\psi_L
	\right) 
	\in 
	\C^{N_h \times L}.
\end{equation}
This matrix encodes the change of basis from the
reduced basis to the FE basis.
Due to \eqref{eq: generic FE formulation}, the reduced basis problem
\eqref{eq:generic_rb_var_form} can be expressed as follows: 
Given $\y \in \truncparamspace$, we seek $\mathbf{u}^{\text{(rb)}}(\y) \in \C^{L}$ such that
\begin{equation}\label{eq: algebraic reduced basis formulation}
	\mathbb{V}^{\text{(rb)}\dagger}_L
	\mathbf{G}_{h}
	\left(
		\mathbb{V}^{\text{(rb)}}_L \mathbf{u}^{\text{(rb)}}_L(\y); \y
	\right) 
	= 
	\boldsymbol{0}
	\in \mathbb{C}^L,
\end{equation}
where $\dagger$ denotes the Hermitian conjugate of $\mathbb{V}^{\text{(rb)}}_L$.
We will henceforth refer to \eqref{eq: algebraic reduced basis formulation} as 
the Galerkin-POD problem or simply G-POD.

Note that G-POD problem, although a system (either linear or non-linear) of $L$ equations, still
requires the assembly of the full-order model as described in {\cite{hesthaven2016certified,Quarteroni_2016}}.
In special cases, e.g.
whenever the dependence of the underlying form $\mathsf{g}$ on the parametric input $\y \in \text{U} $ is affine, 
it is possible that problem \eqref{eq: algebraic reduced basis formulation} turns out to be independent of $N_h$ \cite{Quarteroni_2016}. However, the Helmholtz problem as well as the time-harmonic Maxwell problem presented above do not satisfy this assumption.

\subsection{Reduced Basis Construction: Proper Orthogonal Decomposition}
\label{sec:POD}
A well-known and straightforward approach to construct a reduced basis is the
\emph{proper orthogonal decomposition} (POD) method.
Suppose we have a collection of $N_s$ evaluations of the full-order model
on a finite parameter set $\Xi_{N_s} = \{\y^{(1)}, \ldots, \y^{(N_s)}\} \subset\truncparamspace$, 
the so-called snapshots $\{u_h(\y^{(1)}), \ldots, u_h(\y^{(N_s)})\}$.

For what follows, we assume that a sufficiently large number of snapshots $N_s$
have been computed so that the associated subspace
\begin{align}
	\mathcal{M}_\Xi
    	\coloneqq
    	\mathrm{span}
	\left\{
		u_h\left(\y^{(1)}\right), \ldots, u_h\left(\y^{(N_s)}\right)
	\right\}
    	\subset
	V_h
\end{align}
is a good approximation of the continuous \emph{solution manifold}
\begin{align}\label{eq:sol_manifold}
	\mathcal{M}_h
	\coloneqq
	\{u_h(\y) \mid \y \in \truncparamspace\}.
\end{align}
Thus, we search for a parameter-independent reduced basis $\{\psi_1, \ldots, \psi_L\}$ for $\mathcal{M}_\Xi$, such that $L \ll N_h$ and such that the reduced basis well approximates $\mathcal{M}_\Xi \subset \mathcal{M}_h$.\\

Let $\mathbb{S} \in \C^{N_h \times N_s}$ denote the \emph{snapshot matrix},
defined as
\begin{align}\label{eq:snapshot_matrix}
   	\mathbb{S}
	\coloneqq
    	\left(
    		\mathbf{u}_h\left(\y^{(1)}\right), \ldots, \mathbf{u}_h\left(\y^{(N_s)}\right)
	\right)
	\in
	\mathbb{C}^{N_h \times N_s}.
\end{align}
Let $R$ be the rank of $\mathbb{S}$.
Then the singular value decomposition yields two unitary matrices
\begin{equation}
	\mathbb{W} = \left(\mathbf{w}_1, \ldots, \mathbf{w}_{N_s}\right)
	\in 
	\C^{N_h \times N_h}
	\quad
	\text{and}
	\quad
	\mathbb{Z} = \left(\mathbf{z}_1, \ldots , \mathbf{z}_{N_s}\right) 
	\in 
	\C^{N_s \times N_s}
\end{equation}
and a diagonal matrix
$\mathbb{D} = \mathrm{diag}(\sigma_1, \ldots, \sigma_R)$, with $\sigma_1 \geq \sigma_2 \geq \ldots \geq \sigma_R > 0$, 
$R$ being the rank of $\mathbb{S}$, such that
\begin{align}\label{eq: svd of snapshot matrix}
	\mathbb{S}
	= 
	\mathbb{W}
	\begin{pmatrix}
        		\mathbb{D} & 0 \\
        		0 & 0
    	\end{pmatrix} 
	\mathbb{Z}^\dagger = \mathbb{W} \Sigma \mathbb{Z}^\dagger.
\end{align}
Algebraically speaking, we want to approximate the columns in $\mathbb{S}$ using $L \leq R$ orthonormal vectors $\{\widetilde{\mathbf{w}}_1, \ldots, \widetilde{\mathbf{w}}_L\}$. 
The orthogonal projection of $\mathbf{u}_h\left(\y^{(n)}\right)$ onto $\mathrm{span}\{\widetilde{\mathbf{w}}_1, \ldots, \widetilde{\mathbf{w}}_L\}$ is given by
\begin{align*}
	\sum_{l = 1}^L
	\left(
		\mathbf{u}_h\left(\y^{(n)}\right), \widetilde{\mathbf{w}}_l
	\right)_{\C^{N_h}} \widetilde{\mathbf{w}}_l.
\end{align*}
We seek an orthonormal basis $\{\widetilde{\mathbf{w}}_1, \ldots, \widetilde{\mathbf{w}}_L\}$
such that the quantity
\begin{align}\label{eq:discrete_E}
	\widetilde{\mathscr{E}}^{(N_s)}
	\left(\mathbb{V}_L\right)
	\coloneqq
	\sum_{n = 1}^{N_s}
	\norm{
		\mathbf{u}_h\left(\y^{(n)}\right)
		- 
		\sum_{\ell = 1}^L
		\left(\mathbf{u}_h\left(\y^{(n)}\right), \widetilde{\mathbf{w}}_\ell \right)_{\C^{N_h}} 
		\widetilde{\mathbf{w}}_\ell
	}^2_{\C^{N_h}},
\end{align}
is minimized. The Schmidt-Eckart-Young theorem \cite{Quarteroni_2016}
asserts that the minimum is achieved for the basis $\{\mathbf{w}_1, \ldots, \mathbf{w}_L\}$
consisting of the first $L$ columns of $\mathbb{W}$ in the SVD of $\mathbb{S}$
in \eqref{eq: svd of snapshot matrix}.
Hence, we set $\boldsymbol \psi_l = \mathbf{w}_l$ for $l = 1,\ldots, L$ and thus
\begin{align}\label{eq:reduced_space_L}
	\mathbb{V}^{\text{(rb)}}_L
    	= 
	\left(
		\mathbf{w}_1, \ldots, \mathbf{w}_L
	\right)
	\in 
	\mathbb{C}^{N_h \times L}.
\end{align}

\subsection{Parametric Holomorphy}\label{eq:parametric_holomorphy}
In Sections \ref{sec:reduced_basis_method} and \ref{sec:POD} we have discussed  
computational aspects of the reduced basis method for parametric problems. 
The reduced basis method, and for that matter any other model
order reduction technique, 
can successfully approximate the solution manifold $\mathcal{M}_h$
defined in \eqref{eq:sol_manifold}, and its discrete counterpart $\mathcal{M}_\Xi$,
provided that there exists an intrinsic low-dimensional dynamics driving the behavior 
of solution manifold. 

A commonly used concept in nonlinear approximation to 
quantify uniform error bounds is the so-called 
Kolmogorov’s width.
For a compact subset $\mathcal{K}$ of a Banach space
$X$ it is defined for $L\in \IN$ as
\begin{align}
	d_L(\mathcal{K},X)
	\coloneqq
	\inf _{\operatorname{dim}\left(X_L\right) \leq L} 
	\sup _{v \in \mathcal{K}} 
	\min _{w \in X_L}
	\norm{v-w}_X,
\end{align}
where the outer infimum is taken over all
finite dimensional spaces $X_L\subset X$ of  
dimension smaller than $L$. %
This quantifies
the suitability of $L$-dimensional subspaces
for the approximation of the solution manifold.
As it has been established, see e.g. \cite{hesthaven2016certified,Quarteroni_2016}
the convergence analysis of the reduced basis method relies on the existence
of bounds controlling the decay of the Kolmogorov's width.

A key insight to establish dimension-independent convergence
of Kolmogorov's width for parametric maps with high-dimensional parametric
inputs corresponds to the analytic or holomorphic dependence of the parameter-to-operator
map upon the parametric variables.%

For $s>1$ we define 
the Bernstein ellipse
\begin{align}
	\mathcal{E}_s
	\coloneqq 
	\left\{
		\frac{z+z^{-1}}{2}: \; 1\leq \snorm{z}\leq s
	\right \} 
	\subset \IC.
\end{align}
This ellipse has foci at $z=\pm 1$ and semi-axes of length 
$a\coloneqq  (s+s^{-1})/2$ and $b \coloneqq  (s-s^{-1})/2$.
In addition, we define the tensorized poly-ellipse
\begin{align}
	\mathcal{E}_{\boldsymbol{\rho}} 
	\coloneqq  
	\bigotimes_{j\geq1} 
	\mathcal{E}_{\rho_j} \subset \IC^{\mathbb{N}},
\end{align}
where $\boldsymbol\rho \coloneqq  \{\rho_j\}_{j\geq1}$
is such that $\rho_j>1$, for $j\in \mathbb{N}$.

\begin{definition}[{\cite[Definition 2.1]{CCS15}}]\label{def:bpe_holomorphy}
Let $X$ be a complex Banach space equipped with the norm $\norm{\cdot}_{X}$. 
For $\varepsilon>0$ and $p\in(0,1)$, we say that map 
$\fullparamspace  \ni  \y \mapsto  u(\y)  \in  X$
is $(\boldsymbol{b},p,\varepsilon)$-holomorphic if and only if:
\begin{enumerate}
	\item\label{def:bpe_holomorphy1}
	The map $\fullparamspace \ni {\y} \mapsto u(\y) \in X$ is uniformly bounded.
	\item\label{def:bpe_holomorphy2}
	There exists a positive sequence $\boldsymbol{b}\coloneqq \{b_j\}_{j\geq 1} \in \ell^p(\mathbb{N})$ 
	and a constant $C_\varepsilon>0$ such that for any sequence 
	$\boldsymbol\rho\coloneqq \{\rho_j\}_{j\geq1}$ 
	of numbers strictly larger than one that is
	$(\boldsymbol{b},\varepsilon)$-admissible, i.e.~satisfying
	$\sum_{j\geq 1}(\rho_j-1) b_j  \leq  \varepsilon$,
	the map $\y \mapsto u(\y)$ admits a complex
	extension $\z \mapsto u(\z)$ 
	that is holomorphic with respect to each
	variable $z_j$ on a set of the form 
	\begin{align}
		\mathcal{O}_{\boldsymbol\rho} 
		\coloneqq  
		\displaystyle{\bigotimes_{j\geq 1}} \, \mathcal{O}_{\rho_j},
	\end{align}
	where $ \mathcal{O}_{\rho_j}=
	\{z\in\IC\colon\operatorname{dist}(z,[-1,1])<\rho_j-1\}$.
	\item
	This extension is bounded on $\mathcal{E}_{\boldsymbol\rho}$ according to
	$ \sup_{\z\in \mathcal{E}_{\boldsymbol{\rho}}} \norm{u(\z)}_{X}  \leq C_\varepsilon$.
\end{enumerate}
\end{definition}
The following result addresses  the holomorphic dependence of the solution to 
both problems upon the parametric input. 

\begin{proposition}[Parametric Holomorphy of the Discrete Parameter-to-Solution Map]
\label{prop:parametric_holomorphy}
Let Assumption \ref{assumption:parametric_holomorphy} 
be satisfied with $p \in (0,1)$ and $\boldsymbol{b} \in \ell^p(\IN)$.
\begin{itemize}
	\item[(i)]
	{\bf Helmholtz Impedance Problem}.
	The map 
	$
		\mathcal{S}_{\normalfont\text{Helmholtz}}:
		\fullparamspace
		\rightarrow
		H^1({\text{D}}_0):
		\y
		\mapsto
		\widehat{u}_{h}(\y)
	$
	is $(\boldsymbol{b},p,\varepsilon)$-holomorphic for 
	some $\varepsilon>0$.
	\item[(ii)]
	{\bf Maxwell Lossy Cavity}.
	The map 
	$
		\mathcal{S}_{\normalfont\text{Maxwell}}:
		\fullparamspace
		\rightarrow
		H_0 \left(\Curl;{\text{D}}_0\right):
		\y
		\mapsto
		\widehat{\boldsymbol{E}}_{h}(\y)
	$
	is also $(\boldsymbol{b},p,\varepsilon)$-holomorphic for 
	some $\varepsilon>0$.
\end{itemize}
In either case, $\varepsilon>0$ does not depend on the Galerkin discretization parameter $h>0$.
\end{proposition}

\begin{proof}
A complete proof of the first statement is included in Appendix \ref{appendxi:proof_holomorphy}.
The second statement has been proved in \cite{Aylwin2020}.
\end{proof}

Let us define the solution manifold for the discrete Helmholtz impedance and
Maxwell lossy cavity problems as
\begin{equation}\label{eq:solution_manifold}
\begin{aligned}
	\mathcal{M}_{\normalfont\text{Helmholtz}}
	&
	\coloneqq
	\left\{
		\widehat{u}_{h}(\y)
		\in H^1(\text{D}_0):
		\y \in \text{U}
	\right\},
	\quad
	\text{and} \\
	\mathcal{M}_{\normalfont\text{Maxwell}}
	&
	\coloneqq
	\left\{
		\widehat{\boldsymbol{E}}_{h}(\y)
		\in H_0 \left(\Curl;{\text{D}}_0\right):
		\y \in \text{U}
	\right\}.
\end{aligned}
\end{equation}

As thoroughly discussed in \cite{CD16}, as a consequence of
Proposition \ref{prop:parametric_holomorphy} we have the following
result concerning the decay of Kolmogorov's width for the solution manifolds
introduced in \eqref{eq:solution_manifold}.

\begin{lemma}[Decay of Kolmogorov's Width, \cite{CD16}]\label{lmm:kolmogorov}
Let Assumption \ref{assumption:parametric_holomorphy} be satisfied
with $\boldsymbol{b} \in \ell^p(\IN)$ and $p\in (0,1)$. Then, it holds
\begin{equation}
\begin{aligned}
	d_L\left(\mathcal{M}_{\normalfont\text{Helmholtz}},H^1(\normalfont\text{D}_0)\right)
	&
	\leq
	C(L+1)^{
	-
	\left(
	\frac{1}{p}-1
	\right)
	}
	\quad
	\text{and} \\
	d_L\left(\mathcal{M}_{\normalfont\text{Maxwell}},H_0 \left(\Curl;{\normalfont\text{D}}_0\right))\right)
	&
	\leq
	C(L+1)^{
	-
	\left(
	\frac{1}{p}-1
	\right)
	},
\end{aligned}
\end{equation}
for some constant $C>0$ independent of $L \in \mathbb{N}$.
\end{lemma}

\subsection{Convergence of the Galerkin-POD RB Method}
\label{sec:convergence_rb_method}
In Section \ref{eq:parametric_holomorphy} we established the holomorphic dependence
of the solution to both the Helmholtz interior impedance and Maxwell lossy cavity problems
upon the parametric variables describing the problems' shape deformations.
Among the consequences of this property, and relevant for subsequent developments, we have 
the approximability of the solution manifolds defined in \eqref{eq:solution_manifold} through
finite dimensional linear subspaces. This property is reflected in terms of the dimension-independent decay of Kolmogorov's width as described in Lemma \ref{lmm:kolmogorov}.

Using the properties of the solution manifold described in Section \ref{eq:parametric_holomorphy},
and following the presentation of \cite{hesthaven2016certified,Quarteroni_2016}, we are interested in
establishing dimension-independent convergence rates of the Galerkin-POD-RB method.

To this end, we observe that the $(\boldsymbol{b},p,\varepsilon)$-holomorphy
of the parameter-to-solution map implies $u_h\in L^2(\truncparamspace;V_h)$, thus
$u_h$ is a Hilbert-Schmidt kernel and $\mathsf{T}: L^2(\truncparamspace) \rightarrow V_h$ defined as
\begin{equation}
	\mathsf{T} g 
	=
	\int\limits_{\truncparamspace}
	u_h(\y) g(\y) 
	\,
	\text{d} \y,
	\quad
	g \in L^2(\truncparamspace),
\end{equation}
is a compact Hilbert-Schmidt operator with adjoint $\mathsf{T}^\dagger:V_h \rightarrow L^2(\truncparamspace)$
admitting for each $\y \in \fullparamspace$ the following expression
$
	\left(
		\mathsf{T}^\dagger
		v_h
	\right)(\y)
	=
	\dotp{u_h(\y)}{v_h}_V
	\quad
	\forall 
	v_h \in V_h.
$
As a consequence, the operators
$
	\mathsf{K}
	\coloneqq
	\mathsf{T}
	\mathsf{T}^\dagger:
	V_h
	\rightarrow
	V_h
$
and
$
	\mathsf{C}
	\coloneqq
	\mathsf{T}^\dagger
	\mathsf{T}:
	L^2(\truncparamspace)
	\rightarrow
	L^2(\truncparamspace)
$
are Hermitian, non-negative, and compact.
The latter operator can be represented by the
matrix
\begin{equation}
	{\bf K}_h
	=
	\int\limits_{\normalfont\truncparamspace}
	{\bf u}_h(\y)
	{\bf u}_h(\y)^\dagger
	\,
	\text{d}
	\y
	\in \mathbb{C}^{N_h \times N_h}.
\end{equation} 
Let $\sigma^2_{1}\geq\cdots\geq\sigma^2_{r}>0$ be the eigenvalues
of ${\bf K}_h$, with $r = \text{rank}({\bf K}_h)$, associated to the eigenvectors
$\boldsymbol\zeta_1,\dots, \boldsymbol\zeta_r$,
respectively, i.e.~
$
{\bf K}_h
\boldsymbol{\zeta}_i
=
\sigma^2_{i}
\boldsymbol{\zeta}_i,
\quad
i=1,\dots,r
$. Let us set for $i=1,\dots,r$
\begin{equation}
\zeta_{i,h}
=
\sum_{j=1}^{N_h}
\left(
\boldsymbol\zeta_{i}
\right)_j
\varphi_{j}
\in 
V_h
\quad
\text{and}
\quad
{V}^{(\normalfont\text{rb})}_L
=
\normalfont\text{span}
\left\{
\zeta_{1,h}
,
\dots
,
\zeta_{L,h}
\right\},
\end{equation}
Then, according to \cite[Proposition 6.3]{Quarteroni_2016}, it holds
\begin{equation}\label{eq:min_time_domain_2}
	V^{\text{(rb)}}_L
	=
	\argmin_{
	\substack{
		V_L \subset V_h
		\\
		\text{dim}(V_L)\leq L	
	}
	}	
	\norm{
		u_h
		-
		\mathsf{P}_{V_L}
		u_h
	}^2_{L^2(\truncparamspace;V)}.
\end{equation}

For each $\y \in \normalfont\truncparamspace$, 
we are interested in finding $u^{(\text{rb})}_L(\y) \in {V}^{\text{(rb)}}_L$
such that
\begin{equation}\label{eq:solution_rb}
	\mathsf{a}
	\left(
		u^{(\normalfont\text{rb})}_L(\y),v^{(\text{rb})}_L;\y
	\right)
	=
	\ell\left(v^{(\text{rb})}_L;\y\right),
	\quad
	\forall v^{(\text{rb})}_L \in {V}^{(\normalfont\text{rb})}_L.
\end{equation}
In the following, we refer to $\widehat{u}^{(\normalfont\text{rb})}_L$
and $\widehat{\bE}^{(\normalfont\text{rb})}_L$
as the solution of \eqref{eq:solution_rb} when considering the 
reduced counterparts of the discrete Helmholtz impedance and 
the Maxwell cavity problems as described in Subsection \ref{sec:discrete_FOM_Helmholtz}
and \ref{sec:discrete_FOM_Maxwell}, respectively.

\begin{theorem}[Convergence of the Galerkin-POD RB Method]
Let Assumption \ref{assumption:parametric_holomorphy} 
be satisfied with $p \in (0,1)$ and $\boldsymbol{b} \in \ell^p(\IN)$.
\begin{itemize}
	\item[(i)]
	{\bf Helmholtz Impedance Problem.}
	There exists $L_0 \in \IN$ such that there exists $C>0$ such that for each $L\geq L_0$
	and any $J\in \mathbb{N}$ it holds
	\begin{equation}
		\norm{
			\widehat{u}_h
			-
			\widehat{u}^{(\normalfont\text{rb})}_L
		}_{L^2\left(\truncparamspace;H^1\left({\normalfont\text{D}}_0\right)\right)}
		\leq
		C(L+1)^{
		-	
		\left(
		\frac{1}{p}-1
		\right)
		}.
	\end{equation}
	\item[(ii)]
	{\bf Maxwell Lossy Cavity.} There exists $C>0$ and $L_0\in \mathbb{N}$
	such that for each $L \in \IN$ and any $J\in \mathbb{N}$
	\begin{equation}
		\norm{
			\widehat{\bE}_h
			-
			\widehat{\bE}^{(\normalfont\text{rb})}_L
		}_{L^2\left(\truncparamspace;H_0\left(\Curl;{\normalfont\text{D}_0}\right)\right)}
		\leq
		C(L+1)^{
		-	
		\left(
		\frac{1}{p}-1
		\right)
		}.
	\end{equation}
\end{itemize}
\end{theorem}

\begin{proof}
Firstly, we consider the Helmholtz interior impedance problem.
Similarly as for the discrete Helmholtz full order model described in
\ref{sec:discrete_FOM_Helmholtz}, the application of Cea's Lemma
(which in principle is valid for any finite dimensional subspace) yields
for each $\y \in \text{U}$ and $L\geq L_0$ 
\begin{equation}
	\norm{
		\widehat{u}_h(\y)
		-
		\hat{u}^{(\normalfont\text{rb})}_L(\y)
	}_{H^1\left({\normalfont\text{D}}_0\right)}
	\leq
	C
	\inf_{v_L \in{V}^{(\normalfont\text{rb})}_L.}
	\norm{
		\widehat{u}_h(\y)
		-
		v_L
	}_{H^1\left({\normalfont\text{D}}_0\right)},
\end{equation}
where $C>0$ is a uniform constant and a minimal level of
refinement $L_0$ of the reduced space is required. The final assertion
follows from Lemma \ref{lmm:kolmogorov} and \eqref{eq:min_time_domain_2}.
The assertion for the Maxwell cavity problem follows from the exact same arguments, however
no base level of refinement of the reduced space is required due to the ellipticity of the
corresponding sesquilinear form, i.e.~\eqref{eq:proof_coercivity_Maxwell}.
\end{proof}

\subsection{Snapshot Selection}
The results presented rely in Section \ref{sec:convergence_rb_method} 
on the assumption that the reduced basis 
$V^{\text{(rb)}}_L$ can be computed as in \eqref{eq:min_time_domain_2},
which in turn implies the exact computation of an integral over the parameter
space $\truncparamspace$. Even after considering only the first $J$ parametric dimensions,
this integral needs to be approximated by a suitable
quadrature rule in $\text{U}^{(J)} \coloneqq [-1,1]^J$, as discussed, e.g.,~in \cite[Section 6.5]{Quarteroni_2016}.
The quadrature points define the set $\Xi_{N_s}$ introduced in Section \ref{sec:POD}
for the computation of the snapshots. The effect of the truncation in the parametric dimension
yields as error term decaying as $J^{-\left(\frac{1}{p}-1\right)}$.

Consider the general framework introduced in Section \ref{sec:discrete_full_order_model}.
For a finite dimensional subspace $V_L  = \text{span}\{v_1,\dots,v_L\} \subset V_h$ we set
\begin{equation}
	\mathscr{E}(V_L)
	=
	\norm{
		u_h
		-
		\mathsf{P}_{V_L}
		u_h
	}^2_{L^2\left(\truncparamspace;V\right)}
	\quad
	\text{and}
	\quad
	\mathscr{E}^{(N_s)}(V_L)
	=
	\frac{1}{N_s}
	\sum_{i=1}^{N_s}
	\norm{
		u_h(\y^{(i)})
		-
		\mathsf{P}_{V_L}
		u_h(\y^{(i)})
	}^2_{V},
\end{equation}
where the latter is an $N_s$-points, equal weights,
$J$-dimensional quadrature rule
with quadrature points $\left\{\y^{(1)}, \ldots, \y^{(N_s)}\right\} \subset\truncparamspace$ which approximates
$\mathscr{E}(V_R)$.
As in \cite[Section 6.5]{Quarteroni_2016}, we decompose the error as follows
\begin{equation}\label{eq:quadrature_error}
	\mathscr{E}(V_L)
	\leq
	\underbrace{
	\snorm{
		\mathscr{E}(V_L)
		-
		\mathscr{E}^{(N_s)}(V_L)
	}
	}_{\text{Quadrature Error}}
	+
	\mathscr{E}^{(N_s)}(V_L)
\end{equation}
Furthermore, as we are working in a finite dimensal subspace $V_h$, one has
\begin{equation}
	\mathscr{E}^{(N_s)}(V_L)
	\cong
	\widetilde{\mathscr{E}}^{(N_s)}
	\left(\mathbb{V}_L\right),
\end{equation}
where the hidden constants depend on $V_h$, $\mathbb{V}_L = ({\bf v}_1,\dots,{\bf v}_L) \in \mathbb{C}^{N_h \times L}$,
$v_j$ and ${\bf v}_j$ are connected as described in \eqref{eq:rel_dof_FE},
and $\widetilde{\mathscr{E}}^{(N_s)}$ is as in \eqref{eq:discrete_E}.

The quadrature error in \eqref{eq:quadrature_error} depends on the
problem's parametric dimension. 
The parametric domain deformations considered in this work, as
described in Section \ref{sec: A family of domain mappings},
allowing high-dimensional parametric inputs controlling the domain's
shape deformations.  

In the following, we consider low discrepancy sequences as
quadrature rules. For a specific definition we refer to \cite{caflisch1998monte}.
Examples of low-discrepancy sequences
include those of Sobol' \cite{sobol1967distribution}, Halton \cite{halton1960efficiency}, and Owen \cite{owen1997monte}.
In \cite{lye2020deep,mishra2021enhancing}, low-discrepancy sequences have been considered
for the generation of training data in the approximation of quatities of interested by means of NNs.
The exact same principle applies for the approximation of the quadrature error in \eqref{eq:quadrature_error}.
Indeed, as in \cite[Lemma 3.4]{mishra2021enhancing} one can show that
\begin{equation}
	\snorm{
		\mathscr{E}(V_L)
		-
		\mathscr{E}^{(N_s)}(V_L)
	}
	\leq
	C
	\text{V}_{\text{HK}}
	\frac{(\log N_s)^J}{N_s},
\end{equation}
for a constant $C>0$, where $\text{V}_{\text{HK}}$ corresponds to the Hardy-Krause variation
of the map 
$\truncparamspace 
	\ni 
	\y \mapsto \norm{u_h(\y)-\mathsf{P}_{V_L} u_h(\y)}^2_{V} \in 	\mathbb{R},
$
. Therefore, we may conclude that for
$\mathbb{V}^{\text{(rb)}}_L$ as in \eqref{eq:reduced_space_L} and with 
${V}^{\text{(rb)}}_L$ being the representation of this basis in the FE space $V_h$,
together with \ref{eq:quadrature_error} we have
\begin{equation}
	\mathscr{E}({V}^{\text{(rb)}}_L)
	\lesssim
	\text{V}_{\text{HK}}
	\frac{(\log N_s)^J}{N_s}
	+
	\sum_{j=L+1}^R\sigma^2_j,
\end{equation}
where $\sigma_j>0$ are the singular values of the snapshot matrix $\mathbb{S}$
defined in \eqref{eq:snapshot_matrix}.

\subsection{Centered RB-POD Implementation}
\label{sec:centered_rb_construction}
As in \cite{Chen_2021}, we consider a construction of the reduced basis
in the following referred to as the \emph{centered RB-POD}.
Provided snapshots $\mathbf{s}_1,\dots,\mathbf{s}_{N_s} \in  \mathbb{C}^{N_s}$,
we define the mean of the snapshots as
\begin{equation}
	\overline{\mathbf{u}}
	\coloneqq
	\frac{1}{N_s} \sum_{n = 1}^{N_s} \mathbf{u}_h\left(\y^{(n)}\right) \in \C^{N_h}.
\end{equation}
Let $\mathbb{S}$ be the snapshot matrix as in \eqref{eq:snapshot_matrix},
set $\overline{\mathbb{S}} = (\overline{\mathbf{u}}, \dots,\overline{\mathbf{u}}) \in \mathbb{C}^{N_h \times N_s}$,
and consider the SVD of $\mathbb{S} - \overline{\mathbb{S}}$
\begin{align}
	\mathbb{S} - \overline{\mathbb{S}}
	= 
	\overline{\mathbb{W}}
	\begin{pmatrix}
        		\overline{\mathbb{D}} & 0 \\
        		0 & 0
    	\end{pmatrix} 
	\overline{\mathbb{Z}}^\dagger 
	= 
	\overline{\mathbb{W}}\,\overline{\Sigma} \, \overline{\mathbb{Z}}^\dagger.
\end{align}
with $R = \text{rank}(\mathbb{S} - \overline{\mathbb{S}})$, $\overline{\mathbb{D}} \in \IR^{R \times R}$
a diagonal matrix containing the singular values of $\mathbb{S} - \overline{\mathbb{S}}$, and 
\begin{equation}
	\overline{\mathbb{W}} = \left(\overline{\mathbf{w}}_1, \ldots, \overline{\mathbf{w}}_{N_s}\right)
	\in 
	\C^{N_h \times N_h}
	\quad
	\text{and}
	\quad
	\overline{\mathbb{Z}} = \left(\overline{\mathbf{z}}_1, \ldots ,\overline{\mathbf{z}}_{N_s}\right) 
	\in 
	\C^{N_s \times N_s}.
\end{equation}
We set
\begin{align}
	\overline{\V}^{\text{(rb)}}_L 
    	= 
	\left(
		\overline{\mathbf{w}}_1, \ldots, \overline{\mathbf{w}}_L
	\right)
	\in 
	\mathbb{C}^{N_h \times L}.
\end{align}
We look for a reduced solution to \eqref{eq: algebraic reduced basis formulation}
of the form $\overline{\V}^{\text{(rb)}}_L \overline{\mathbf{u}}^{\text{(rb)}}_L(\y) +\overline{\mathbf{u}}$,
thus yielding the following problem: For each $\y \in \truncparamspace$, we seek $\overline{\mathbf{u}}^{\text{(rb)}}_L \in \C^L$ such that
$$
	\overline{\V}^{\text{(rb)}\dagger}_L
	\mathbf{G}_{h}
	\left(	
		\overline{\V}^{\text{(rb)}}_L\overline{\mathbf{u}}^{\text{(rb)}}_L(\y)
		+
		\overline{\mathbf{u}}; 
		\y
	\right) 
	= 
	\boldsymbol{0} \in \C^{L},
$$
which in the linear case amounts to adapting the RHS of the system, as follows
\begin{equation}
	\overline{\V}^{\text{(rb)}\dagger}_L \mathbf{G}_{h}(\y)\overline{\V}^{\text{(rb)}}_L \overline{\mathbf{u}}^{\text{(rb)}}_L(\y) 
	=
	-
	\overline{\V}^{\text{(rb)}\dagger}_L \mathbf{G}_{h}(\y) 
	\overline{\mathbf{u}}.
\end{equation}

%% file: 4_neural_networks.tex
\section{Galerkin POD - Neural Network}
\label{sec: neural network POD}
In this section, we introduce the Galerkin POD Neural Network (POD-NN) as proposed in 
\cite{Hesthaven2018} and propose minor modifications to accommodate the complex-valued
nature of the solutions.
We further formulate the learning problem centered around the mean,
which facilitates the learning task for the NN.

\subsection{Neural Networks}
\label{sec: artificial neural networks}
In this work, we consider multi-layer perceptron architectures consisting of  $D \in \IN$ layers, with layer width  $\ell_0,\dots, \ell_D \in \IN$.  The activation function $\sigma:\IR \rightarrow\IR$ may be chosen as any nonlinear function, we restrict our discussion to the hyperbolic tangent defined as
\begin{align}
	\sigma(x)
	=
	\text{tanh}(x)
	=
	\frac{\exp(x)-\exp(-x)}{\exp(x)+\exp(-x)},
	\quad
	x\in \mathbb{R}.
\end{align}
Given weights and biases $\boldsymbol{\theta} \coloneqq({\bf W}_k,{\bf b}_k)_{k=1}^{D}$,
${\bf W}_k\in\IR^{\ell_k\times\ell_{k-1}}$, ${\bf b}_k\in\IR^{\ell _k}$, we
define the affine transformation
 ${\bf A}_k: \IR^{\ell_{k-1}}\rightarrow \IR^{\ell_k}: {\bf x} \mapsto {\bf W}_k {\bf x}+{\bf b}_k$
for $k\in\{1,\ldots,D\}$. We may then define a neural network 
with activation function $\sigma$ as a map
$\Psi^{\mathcal{N\!N}}_{\boldsymbol{\theta}}: \IR^{\ell_0}\rightarrow \IR^{\ell_D}$ with
\begin{align}\label{eq:ann_def}
	\Psi^{\mathcal{N\!N}}_{\boldsymbol{\theta}}({\bf x})
	\coloneqq
	\begin{cases}
	{\bf A}_1({\bf x}), & D=1, \\
	\left(
		{\bf A}_L
		\circ
		\sigma
		\circ
		{\bf A}_{L-1}
		\circ
		\sigma
		\cdots
		\circ
		\sigma
		\circ
		{\bf A}_1
	\right)({\bf x}),
	& D\geq2,
	\end{cases}
\end{align}
where the activation function $\sigma:\IR\rightarrow \IR$
is applied component-wise to vector-valued inputs.
We define the depth and the width
of an NN as
\[
\normalfont\text{width}\left(\Psi^{\mathcal{N\!N}}_{\boldsymbol{\theta}}\right)=\max\{\ell_0,\ldots, \ell_D\}
\quad
\text{and}
\quad
\normalfont\text{depth}\left(\Psi^{\mathcal{N\!N}}_{\boldsymbol{\theta}}\right)=D.
\]
We denote by $\mathcal{N\!N}_{D,H,\ell_0,\ell_D}$
the set of all NNs $\Psi^{\mathcal{N\!N}}_{\boldsymbol{\theta}}({\bf x}):\mathbb{R}^{\ell_0}\rightarrow \mathbb{R}^{\ell_D}$ 
with input dimension $\ell_0$, output dimension $\ell_D$, a width of at most $H$, and a depth of at most $D$ layers.

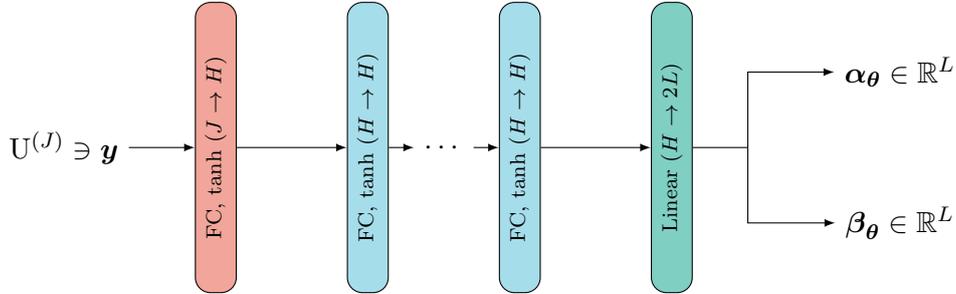
\begin{figure}[htp]
\centering
\input{tikz/nn}
\caption{NN architecture for the approximation of the map $\boldsymbol{\pi}^{\text{(rb)}}_{L,\mathbb{R}}: \truncparamspace \to \R^{2L}$
as in \eqref{eq:ann_parametric_real}.
The NN accepts as input $J$ values accounting for the components of the parametric input $\y = (y_1,\dots,y_J) \in  \truncparamspace$, whereas
there are $2L$ outputs representing both the real and imaginary parts of the reduced coefficients. The input (red) and hidden layers (blue) are fully connected (FC) with hyperbolic tangent ($\tanh$) activation functions. }\label{fig:ANN_complex_decomp}
\end{figure}

\subsection{Formulation of the Learning Problem}
Following the centered RB construction described in Section \ref{sec:centered_rb_construction},
we would like to approximate the parametric map
\begin{align}\label{eq: complex parameter to reduced coeff map}
	\boldsymbol{\pi}^{\text{(rb)}}_L
	: \truncparamspace  \to \C^L:
	\y \mapsto 
	\V^{\text{(rb)}\dagger}_L
	\left(
		{\mathbf{u}}_h(\y) -\overline{\mathbf{u}}
	\right),
\end{align} 
by an NN, where $L\in \IN$ is the dimension of the reduced space $\V^{\text{(rb)}}_L$.

The map introduced in \eqref{eq: complex parameter to reduced coeff map}
has an output that is complex-valued, as the reduced coefficients are complex-valued 
themselves. However, this does not fit the NN definition stated in Section
\ref{sec: artificial neural networks}.
Consequently, we proceed to formulate an
equivalent real-valued learning problem.

Then, instead of approximating the map $\boldsymbol{\pi}^{\text{(rb)}}_L$
as in \eqref{eq: complex parameter to reduced coeff map}, we consider the map
\begin{align}\label{eq:ann_parametric_real}
    \boldsymbol{\pi}^{\text{(rb)}}_{L,\mathbb{R}}:
    \truncparamspace
 \to \R^{2L}:
    \y \mapsto
    \begin{pmatrix}
        \realout(\y) \\
        \imagout(\y)
    \end{pmatrix}
    \coloneqq
    \begin{pmatrix}
	\Re\left\{ 	\V^{\text{(rb)}\dagger}_L
	\left(
		{\mathbf{u}}_h(\y) -\overline{\mathbf{u}}
	\right)\right \} \\
	\Im\left\{ 	\V^{\text{(rb)}\dagger}_L
	\left(
		{\mathbf{u}}_h(\y) -\overline{\mathbf{u}}
	\right)\right \}
    \end{pmatrix},
\end{align}
$\realout(\y),  \imagout(\y) \in \mathbb{R}^L$ for each $\y \in \text{U}^{(L)}$, 
which returns separately the real and imaginary parts of the reduced coefficients
in a real-valued vector of size $2L$.

Given a data set consisting of $N_s$ training inputs $\y^{(i)} \in \text{U}^{(J)}, i = 1, \ldots, N_s$
and the corresponding high-fidelity snapshots ${\bf u}_h\left(\y^{(i)}\right), i = 1, \ldots, N_s$,
we can train an NN $ \boldsymbol{\pi}^{\text{(rb)}}_{\boldsymbol{\theta}} \in \mathcal{N\!N}_{H,D,J,2L}$ 
(as in Section \ref{sec: artificial neural networks}),
i.e.~with $J$ inputs (one for each component of the parametric input $\y \in \fullparamspace$, $2L$ outputs
accounting for the $L$ complex reduced coefficients, and depth and width $D$ and $H$, respectively, 
on the set of training input-output pairs
$
	\mathcal{P}_{\text{Train}} 
	=
	\left\{
		\left(\y^{(i)}, \boldsymbol{\pi}^{\text{(rb)}}_{L,\mathbb{R}}\left(\y^{(i)}\right)\right)
	\right\}_{1 \leq i \leq N_s}.
$
Figure \ref{fig:ANN_complex_decomp} portrays the previously described architecture. The 
first $L$ outputs emulate the real part of the reduced coefficients, whereas the second $L$
outputs account for the imaginary one. 

The mean square error (MSE) is a natural candidate for a loss function.
More precisely, let $\boldsymbol \theta$ denote the vector gathering all weights and biases of the NN
$\boldsymbol{\pi}^{\text{(rb)}}_{\boldsymbol{\theta}}$.
Then, the MSE loss is given by

\begin{equation}\label{eq: MSE_loss_nn}
\begin{aligned}
	L_{\text{MSE}}(\boldsymbol \theta) 
	&
	\coloneqq 
	\frac{1}{N_s}\sum_{i = 1}^{N_s}
	\norm{
		 \boldsymbol{\pi}^{\text{(rb)}}_{L,\mathbb{R}}\left(\y^{(i)}\right) 
		 - 
		\boldsymbol{\pi}^{\text{(rb)}}_{\boldsymbol{\theta}}\left(\y^{(i)}\right)
	}^2_2
	\\
	&
	=
	\frac{1}{N_s}\sum_{i = 1}^{N_s}
	\sum_{l = 1}^{L}
	\snorm{
		\alpha_{l}\left(\y^{(i)}\right)
		-
		\alpha_{l,\boldsymbol\theta}\left(\y^{(i)}\right)
	}^2
	+
	\snorm{
		\beta_{l}\left(\y^{(i)}\right)
		-
		\beta_{l,\boldsymbol\theta}\left(\y^{(i)}\right)
	}^2,
\end{aligned}
\end{equation}
where the outputs of the NN $\boldsymbol{\pi}^{\text{(rb)}}_{\boldsymbol{\theta}}$ are organized
as follows (cp.~Figure \ref{fig:ANN_complex_decomp})

\begin{equation}
	\boldsymbol{\pi}^{\text{(rb)}}_{\boldsymbol{\theta}}(\y)
	=
	\begin{pmatrix}
        \realout_{\boldsymbol{\theta}}(\y)^\top,  \imagout_{\boldsymbol{\theta}}(\y)^\top
	\end{pmatrix}^\top,
	\quad
	\y \in \truncparamspace.
\end{equation}
Let $\boldsymbol{\theta}^\star$ be such that 
\begin{equation}
	\boldsymbol{\theta}^\star
	\in
	\displaystyle\argmin_{\boldsymbol\theta} L_{\text{MSE}}(\boldsymbol \theta).
\end{equation}
Then, the reduced basis solution lifted to the original FEM space can be reconstructed as follows
\begin{equation}
	{\bf u}^{\mathcal{N\!N}}(\y)
	\coloneqq
	\overline{\V}^{\text{(rb)}}_L \left(\realout_{\boldsymbol{\theta}}(\y) + \imath \imagout_{\boldsymbol{\theta}}(\y) \right)
	+
	\overline{\mathbf{u}}
	\in 
	\mathbb{C}^{N_h},
	\quad
	\y \in \truncparamspace.
\end{equation}

\subsection{Approximation Rates of the Galerkin POD-NN}
The Galerkin POD-NN as originally described in \cite{Hesthaven2018} falls in a ever-increasing
body of work usually referred to as \emph{operator learning}. In particular, a thorough study
of the Galerkin POD-NN has been performed in \cite{lanthaler2023operator}.

The following results claim that by separating the real and imaginary of the reduced coefficients does not break the parametric holomorphy property.
Indeed, this is of key importance in establishing dimension-independent emulation rates for the reduced coefficients.

\begin{lemma}
Let Assumption \ref{assumption:parametric_holomorphy} be satisfied
with $\boldsymbol{b} \in \ell^p(\IN)$ and $p\in (0,1)$.
There exists $J_0 \in \IN$ such that for $J\geq J_0$
and for each $\y \in \fullparamspace$ there exists a unique 
$u^{\normalfont\text{(rb)}}_L(\y) \in {V}^{\normalfont\text{(rb)}}_L$
solution to \eqref{eq:solution_rb}.
In addition, the map $\boldsymbol{\pi}^{\normalfont\text{(rb)}}_{L,\mathbb{R}}:\normalfont\text{U} \to \R^{2L}$
is $(\boldsymbol{b},p,\varepsilon)$-holomorphic.
\end{lemma}

\begin{proof}
As in the proof of Proposition \ref{prop:parametric_holomorphy}, we can argue that the maps
\begin{equation}
    \text{U}
    \ni 
    \y \mapsto
    \widehat{u}_{h}(\y)
    \in
    H^1(\text{D}_0)
    \quad
    \text{and}
    \quad
    \text{U}
    \ni 
    \y \mapsto
    \widehat{\boldsymbol{E}}_{h}(\y)
    \in
    H_0 \left(\Curl;{\text{D}}_0\right).
\end{equation}
are $(\boldsymbol{b},p,\varepsilon)$-holomorphic, therefore the parameter-to-reduced coefficients
are so as well. We observe that as in the proof of
Proposition \ref{prop:parametric_holomorphy} a minimal level of refinement $h_0>0$ is required for this to hold for the Helmholtz impedance problem, and so is the case for the discretization in the reduced space. 

As pointed out previously, for the sake of the implementation, the real and imaginary parts of these maps are approximated separately. However, the application of either the real or imaginary parts to a complex input
is not an holomorphic map itself. Therefore, one can not argue that the compositions of these maps
yields an holomorphic one. In \cite[Lemma A.1]{dolz2023parametric}, it is proved that both the real
and imaginary parts of complex-valued holomorphic function preserve this property, thus yielding the desired result.
\end{proof}

Equipped with this result, together with \cite[Lemma 2.6]{dolz2023parametric} which in turn follows from 
\cite{ABD22}, we may the following approximation result of the Galerkin POD-NN.

\begin{lemma}
Let Assumption \ref{assumption:parametric_holomorphy} be satisfied
with $\boldsymbol{b} \in \ell^p(\IN)$ and $p\in (0,1)$. In addition, assume that
$\boldsymbol{b}$ is strictly decreasing.
For each $n \in \mathbb{N}$ there exists a sequence of tanh NN
$\Psi^{(n)}_{\mathcal{N\!N}} \in \mathcal{N\!N}_{D,H,J,2L}$
and $C>0$ such that
\begin{equation}
	\norm{
		\boldsymbol{\pi}^{\normalfont\text{(rb)}}_{L,\mathbb{R}}
		-
		\Psi^{(n)}_{\mathcal{N\!N}} 
	}_{L^2(\normalfont\text{U}^{(J)};\mathbb{R}^{2L})}
	\leq
	C
	n^{-\left(\frac{1}{p} - \frac{1}{2}\right)}
\end{equation}
with 
$
	D 
	=
	\mathcal{O}(n^2)
$
and
$
	H
	=
	\mathcal{O}
	\left(
		\log_2(n)
	\right)
$.
\end{lemma}

%% file: tikz/nn.tex
\def\nhidden{4}
\def\spacing{2.0}

\begin{tikzpicture}[
    neuron/.style={draw, circle, minimum size=1.2em},
    layer/.style={draw, rectangle, minimum width=10em, minimum height=1em, rounded corners=2mm, font=\scriptsize},
    rotateText/.style={rotate=90}
]
\node[minimum width=2em, align=center] (input) at (-\spacing, 0) {$\mathrm{U}^{(J)} \ni \y$};
\node[layer, rotate=90, fill=myred!50] (embedding) at (0,0) {FC, $\tanh$ ($J \to H$)};
\node[layer, rotate=90, fill=myblue!50] (hidden1) at (\spacing,0) {FC, $\tanh$ ($H \to H$)};

\node[] (dummy) at (\spacing*1.5,0) {$\cdots$};

\node[layer, rotate=90, fill=myblue!50] (hidden2) at (\spacing*2,0) {FC, $\tanh$ ($H \to H$)};
\node[layer, rotate=90, fill=mygreen!50] (linear out) at (\spacing*3, 0) {Linear ($H \to 2L$)};

\node[] (dummy out) at (\spacing*3.5, 0) {};

\node[minimum width=2em] (real output) at (\spacing*4.5, 1) {$\realout_{\boldsymbol{\theta}} \in \R^L$};

\node[minimum width=2em] (imag output) at (\spacing*4.5, -1) {
$\imagout_{\boldsymbol{\theta}} \in \R^L$
};
\draw[->] (input) -- (embedding);
\draw[->] (embedding) -- (hidden1);
\draw[->] (hidden1) -- (dummy);
\draw[->] (dummy) -- (hidden2);
\draw[->] (hidden2) -- (linear out);
\draw[] (linear out) -- (dummy out.center);
\draw[->] (dummy out.center) |- (real output.west);
\draw[->] (dummy out.center) |- (imag output.west);

\end{tikzpicture}

%% file: 5_numerical_results.tex
\section{Numerical experiments}
\label{sec: Numerical results}

\subsection{Numerical implementation}
We use the programming language Julia to conduct our numerical experiments. The high-fidelity FE methods for solving the Helmholtz impedance and Maxwell lossy cavity problem are implemented in \texttt{Gridap.jl} \cite{Badia2020} using Lagrange and Nédélec elements, respectively. After assembly, the linear system is solved by Julia's native linear solver.
For the boundary variations, the parameter space of the affine transformations is sampled using a Halton or Latin Hypercube sequence generated by the library \texttt{QuasiMonteCarlo.jl}. To compute the solution of the Galkerin-POD method, the linear operators are first assembled in \texttt{Gridap.jl}, then projected onto the reduced basis, and solved. The neural network architectures and training implementation are based on the library \texttt{Flux.jl} \cite{Innes2018}.
\subsection{Performance evaluation}
As in \cite{Hesthaven2018}, we consider the following \emph{relative} error measures with respect to a high-fidelity solution ${\mathbf{u}_h(\y)}$ to assess the performance of a model:
\begin{enumerate}
	\item the G-POD relative error
    	\begin{align*}
        		\mathcal{E}_{\mathrm{G}}(L, \y)
		\coloneqq
		\frac{
			\norm{
				{\mathbf{u}_h(\y)}
				- 
				\left(
					\overline{\V}^{\text{(rb)}}_L\overline{\mathbf{u}}^{\text{(rb)}}_L(\y)
					+
					\overline{\mathbf{u}}
				\right)
			}
		}{\norm{\mathbf{u}_h(\y)}},
    	\end{align*}
    	\item the POD-NN relative error
        	\begin{align*}
        		\mathcal{E}_{\mathrm{NN}}(L, \y) 
		\coloneqq 
		\frac{
			\norm{
				{\mathbf{u}_h(\y)}
				-
				\left(
				\overline{\V}^{\text{(rb)}}_L \left(\realout_{\boldsymbol{\theta}}(\y) 
				+ \imath \imagout_{\boldsymbol{\theta}}(\y) \right)
				+
				\overline{\mathbf{u}}
				\right)
				}
		}{\norm{\mathbf{u}_h(\y)}},
    	\end{align*}
   	\item the relative projection error, i.e. the relative error between the reconstruction of the projection of a high-fidelity solution and itself
    	\begin{align*}
        		\mathcal{E}_{\V}(L, \y) 
		\coloneqq
		\frac{
			\norm{
				{\mathbf{u}_h(\y)}
				- 
				\left(
					\overline{\V}^{\text{(rb)}}_L({\mathbf{u}_h(\y)}-\overline{\mathbf{u}})
					+
					\overline{\mathbf{u}}
				\right)
				}
		}{\norm{\mathbf{u}_h(\y)}}
    \end{align*}
\end{enumerate}
Clearly, the latter error is a lower bound to the first two. To analyze the global performance of the model, the averages of the above error measures over  test set are considered.

\subsection{Numerical Results for the Helmholtz Impedance Problem}
We test our implementation on the Helmholtz problem across different hyperparameter settings. We impose impedance boundary conditions as defined in ~\eqref{eq: helmholtz problem with impedance boundary condition}. The following parameter choices are our standard setting unless specified otherwise. The wave number is set to $\kappa=1$, and the boundary variation is parametrized with the Matérn-like decay, with $\nu = 0.5$, $l=0.1$, $\theta=0.1$, and parameter dimension $J = 50$. Figure \ref{fig: graphical results} (a) and (b) show the original mesh on the unit cube, as well as a typical deformation of the boundary and domain. We sample 1024 snapshots from a Halton sequence for the construction of the reduced basis and the training of the network. To obtain an unbiased test set, we sample 512 snapshots from a Latin Hypercube sampling, which is used to evaluate the error measures introduced in the previous section, so that we avoid any positive biases between training and test set.\\
The NN baseline architecture comprises $D=2$ hidden layers with $H=30$ neurons and the $\tanh$  activation function. The network parameters are trained on the loss in equation \eqref{eq: MSE_loss_nn}  for 4000 epochs with the ADAM optimizer using a learning rate of $5e-4$.  Figure \ref{fig: graphical results} (c) and (d) show the high-fidelity and POD-NN solution, respectively, depicting no visible differences between the two.

\begin{figure}
    \centering
    \begin{subfigure}{0.4\textwidth}
    \includegraphics[scale=0.3]{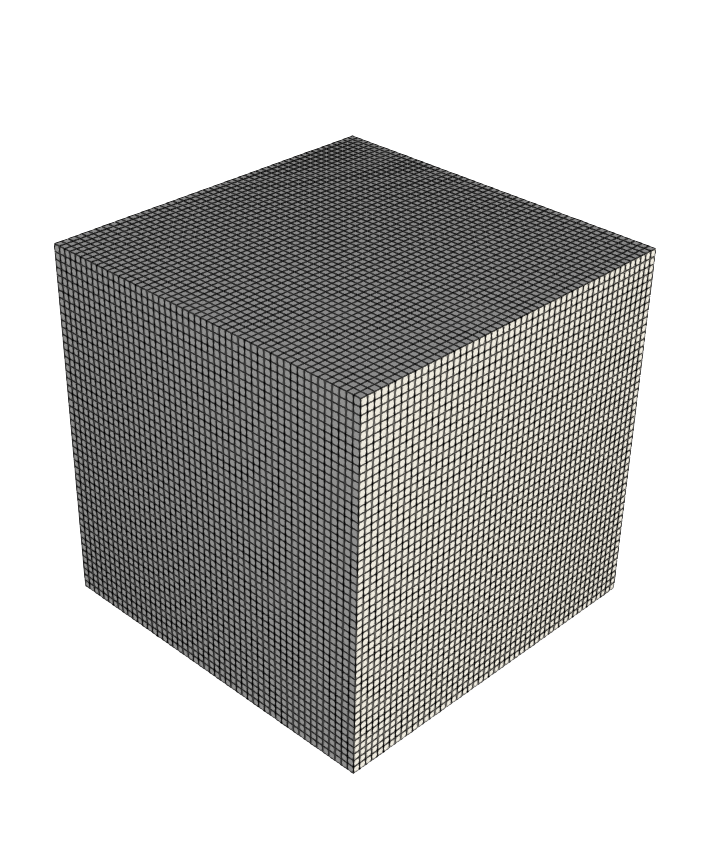}
    \caption{}
    \end{subfigure}
    \hfill
    \begin{subfigure}{0.4\textwidth}
    \includegraphics[scale=0.3]{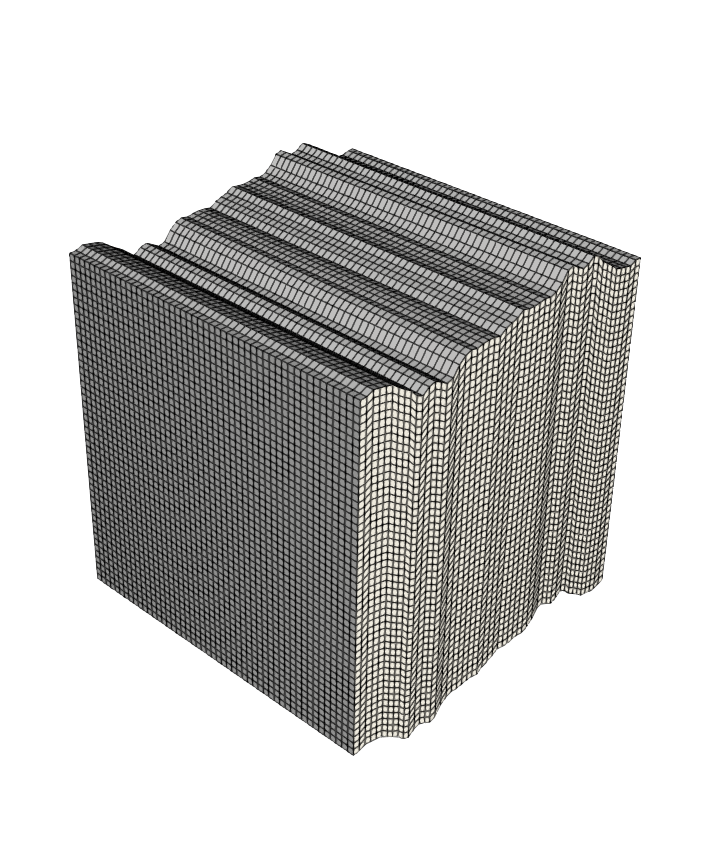}
    \caption{}
    \end{subfigure}
    \hfill
    \begin{subfigure}{0.4\textwidth}
    \includegraphics[scale=0.3]{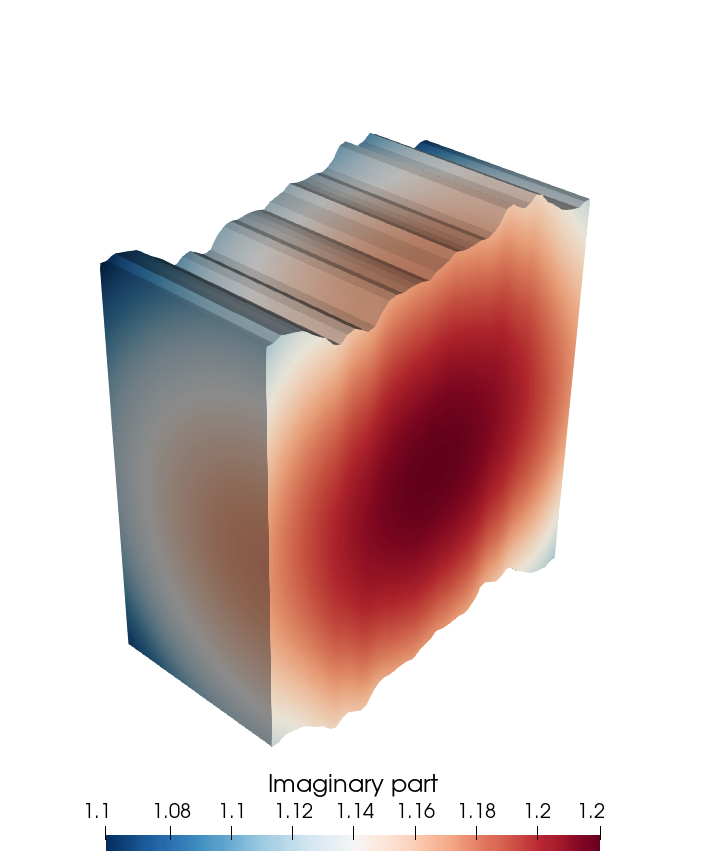}
    \caption{}
    \end{subfigure}
    \hfill
    \begin{subfigure}{0.4\textwidth}
    \includegraphics[scale=0.3]{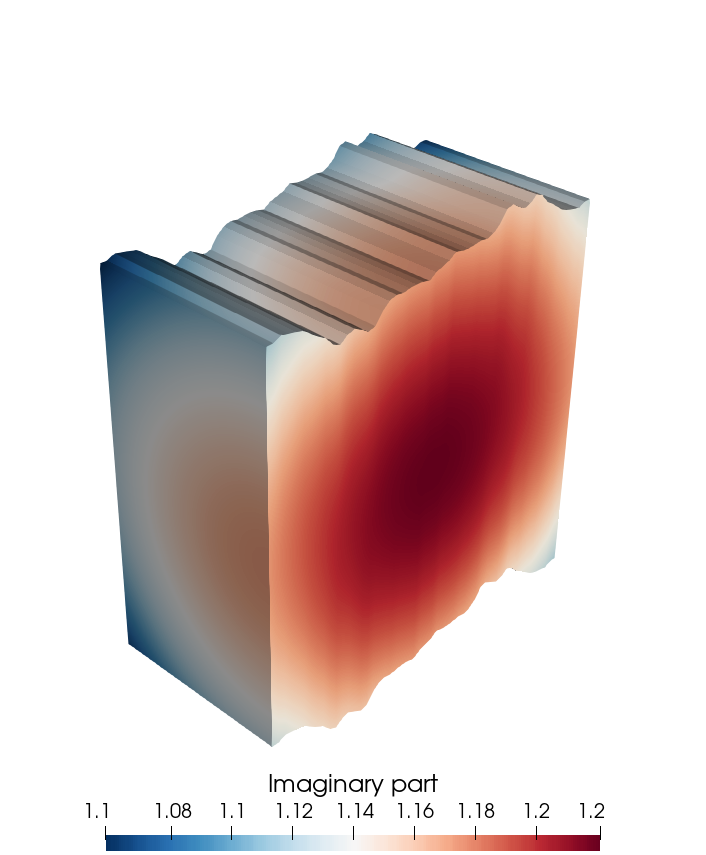}
    \caption{}
    \end{subfigure}
    \hfill
    \caption{Computational meshes and graphical results for the Helmholtz problem: (A) reference mesh, on which the solutions are computed. (B) Physical domain. (C) Imaginary part of the full-order solution. (D) Imaginary part of the POD-NN prediction. The solution to the Helmholtz problem was computed for the parameters $\theta = 0.5, l=0.1, \nu=0.5$, and $J =50$. The domain deformation is amplified by a factor of two for better visibility.}
    \label{fig: graphical results}
\end{figure}

Figure \ref{fig:coeffs} shows the scaling coefficients for the parameters in the 50-dimensional space, particularly the differences between the algebraic and Matérn-like decay. When computing the POD on the assembled snapshot matrix, we further observe that the decay of the singular values (see Figure \ref{fig:svals}) qualitatively follows the trend of the parameter decay. As expected, the fastest algebraic decay $r$=3, also leads to the fastest decay in singular values, while those of the Matérn case decay more slowly. We can thus confirm that an efficient RB construction with a limited basis size is possible for all domain mappings under consideration, which aligns with our theoretical statements in Section 3. 

\begin{figure}
\centering
\captionsetup{width=.48\linewidth}
    \captionbox{The first 50 coefficients $\mu_j$ for different parameters of algebraic (blue tones) and Matérn (red tones) decay.\label{fig:coeffs}}
    [.48\textwidth]{\includegraphics{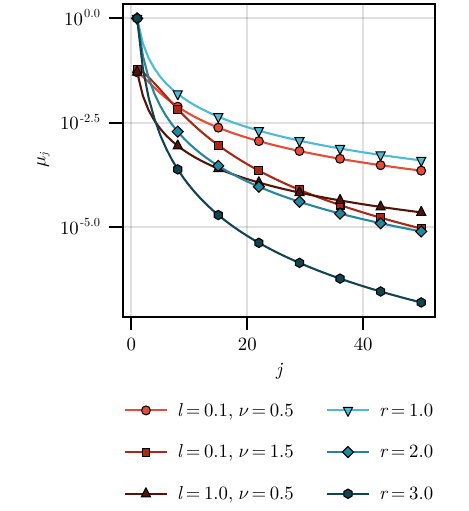}}%
    \hfill
    \captionbox{The singular values of 1024 snapshots for different parameters of algebraic (blue tones) and Matérn (red tones) decay. The input parameters originate from the same Halton sequence.\label{fig:svals}}
    [.48\textwidth]{\includegraphics{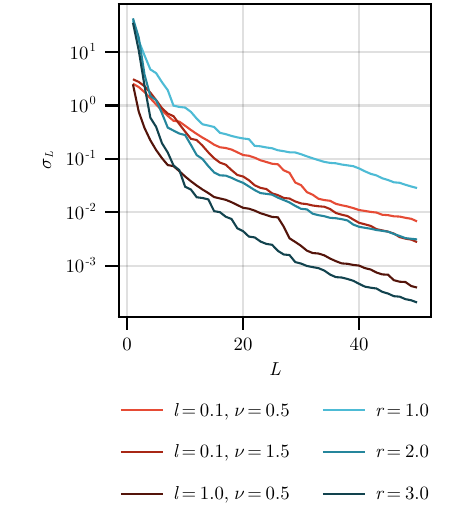}}
\end{figure}

In Figure \ref{fig:architecture comparison}, we assess the test error for different NN architectures and vary the number of modes $L$ in the RB basis. In all figures, we also report the error with zero basis functions ($L$=0), i.e. the error of using the mean field as a predictor. For the standard setting, we observe that the mean field already leads to a low relative error of 3e-4, indicating that the parametric variation in this case is limited. Nevertheless, the best POD-NN architecture further reduces this error by an order of magnitude to 3e-5. We further observe that increasing the depth of the NN does not lead to performance gains. Increasing the width does reduce the error further, but we observe a growth in error after more than 20 basis functions are added for all architectures.  The increase in error is likely due to the fact that the learning problem becomes more difficult by adding more modes, as more RB coefficients need to be approximated: we have verified, that the error increase does not occur when the RB coefficients are approximated by separate NNs. 

\begin{figure}
    \centering
\includegraphics{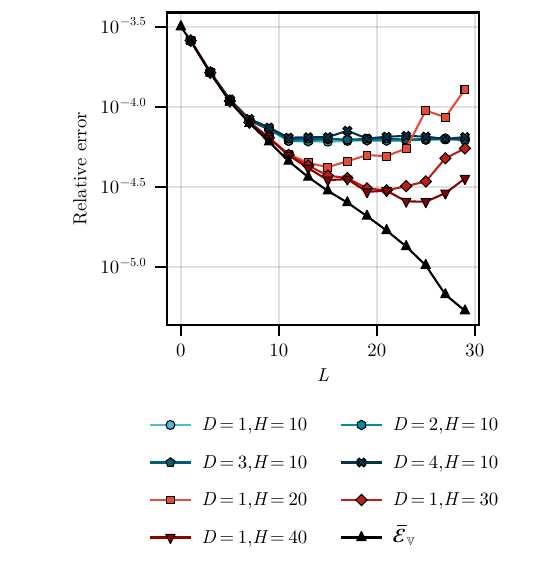}
    \caption{Test errors for different neural network architectures, i.e., different numbers of hidden layers $D$ and neurons per layer $H$. All models were trained using Adam with learning rate 5e-4, $\beta_1 = 0.8$ and $\beta_2 = 0.9$ for 4000 epochs. The networks were trained on 1024 snapshots sampled from a Halton sequence of Matérn decay parameters with the following settings: $\theta = 0.1, J = 50, \nu = 0.5, l = 0.1$.}
    \label{fig:architecture comparison}
\end{figure}

To demonstrate that the developed POD-NN approach may be used across a range of domain mappings, we compare different deformation scalings in Figure \ref{fig:deformation scaling varpar comparison} (left), and unsurprisingly, larger deformations are harder to approximate. Not only is the mean field a worse predictor ($L=$0), both the Galerkin-POD RB method and the POD-NN struggle to decrease the error below 8e-3 for the largest deformation: while the error for $\theta=0.1$ drops below 1e-4, the error for $\theta=0.5$ only drops below 1e-3 for 30 modes, which is only a marginal improvement compared to the error of using only the mean field. That being said, these results were computed for a parameter dimension of size 50, which is a very challenging learning problem. This effect becomes clearly visible in Figure \ref{fig:deformation scaling varpar comparison} (right), where we observe a significantly larger drop in error for parameter dimension size $J=10$. Interestingly, in all cases, we observe a barrier in error for the POD-NN, while the error of the Galerkin-POD may increase intermediately but ultimately keeps decreasing in this case.

\begin{figure}
    \centering
\includegraphics{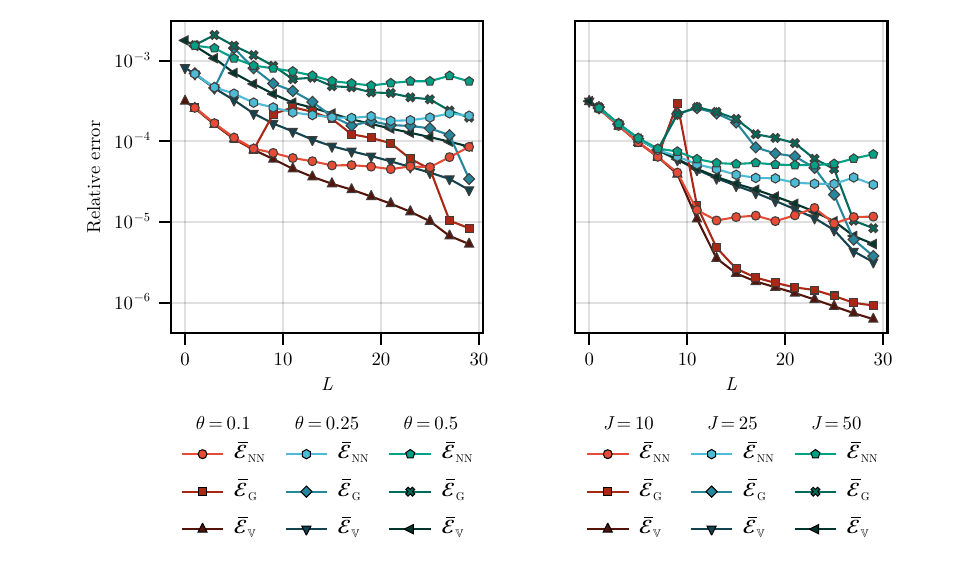}
    \caption{Test errors for deformation scalings (l.) and parameter dimensions (r.). All models were trained using Adam with learning rate 5e-4, $\beta_1 = 0.8$ and $\beta_2 = 0.9$ for 4000 epochs. The networks were trained on 1024 snapshots sampled from a Halton sequence of Matérn decay parameters with the following settings: $\nu = 0.5, l = 0.1$.}
    \label{fig:deformation scaling varpar comparison}
\end{figure}

In Figure \ref{fig:decay comparison}, we compare the effect of the different decays. For the algebraic decay, we seem to be unable to learn meaningful information past the first five modes.
In this case the Galerkin-POD appears to perform much better, which may be attributed to overfitting in the training process as the parameter domain is sampled very sparsely. For the Matérn type decay, we observe a more clear decrease in error for both the POD-NN and the Galerkin-POD. Figure \ref{fig:coeff errors} shows the relative error in the RB coefficient. It is notable that the first seven coefficients have a lower error of around 10$\%$, while the error can increase to up to 50$\%$ for the following coefficients. This is consistent with the barrier in error decay that we have observed: The neural network does not seem to learn useful information past the first eight modes. 

\begin{figure}
    \centering
\includegraphics{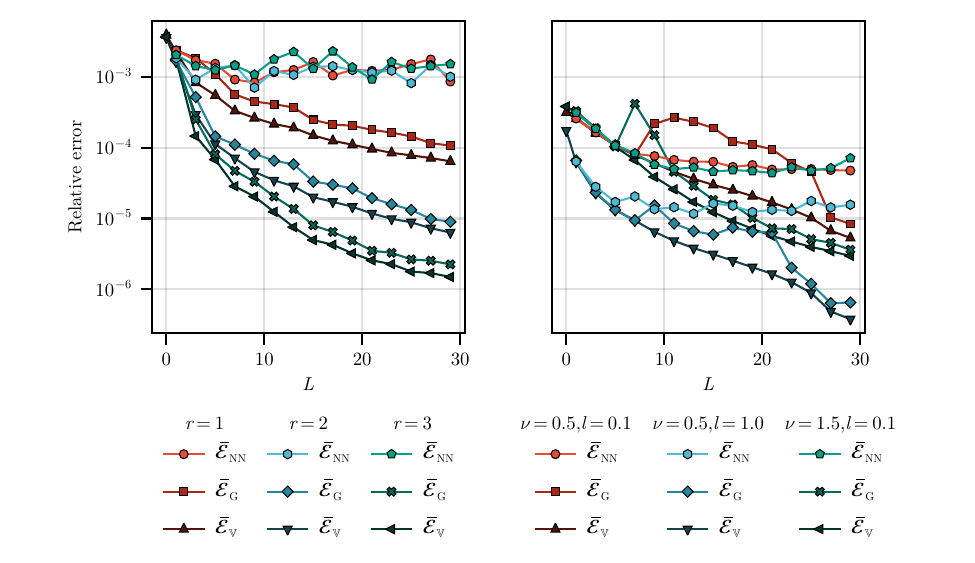}
    \caption{Test errors for algebraic decay (l.) and Matérn decay parameters (r.). All models were trained using Adam with learning rate 5e-4, $\beta_1 = 0.8$ and $\beta_2 = 0.9$ for 4000 epochs. The networks were trained on 1024 snapshots sampled from a Halton sequence of corresponding decay type parameters with the following settings: $\theta = 0.1, J = 50$.}
    \label{fig:decay comparison}
\end{figure}

\begin{figure}
    \centering
    \includegraphics{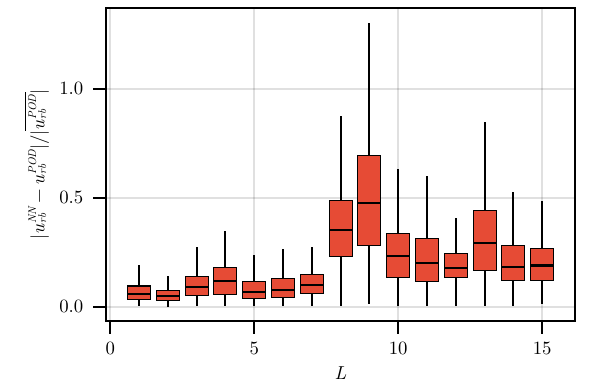}
    \caption{Errors relative to the mean reduced coefficient mode by mode committed by the neural network interpolation. The NN has $D = 2$ hidden layers and $H = 30$ neurons per layer and has been trained with the usual settings on 1024 snapshots (Halton sampling) of Matérn perturbations ($\nu = 0.5$, $l = 0.1$).}
    \label{fig:coeff errors}
\end{figure}

Figure \ref{fig:frequency} explores the effect of different wave numbers $\kappa$. Perhaps unsurprisingly, higher frequencies are harder to approximate. For 
$\kappa$ = 16, the mean field incurs an error of 5e-2, which might be considered a large error for some applications, while the POD-NN and RB method lead to more reliable predictions with an  error of 3e-3, i.e. we again gain about a factor 10 in accuracy.
Lastly, while the POD-NN method has similar or sometimes higher errors than the classic RB method, Figure \ref{fig:speed up} shows that the POD-NN is about a factor of 10000 faster. While the POD-NN method only requires an NN evaluation and some vector operations, the RB method requires full assembly of the linear system operators, which leads to this large discrepancy. In cases, where the mean field alone is not a reliable predictor, the proposed POD-NN method can thus be a valuable and efficient tool to approximate the parameter-to-solution map.

\begin{figure}
\centering
\captionsetup{width=.48\linewidth}
    \captionbox{Test errors for different wavenumbers $\kappa$. All models were trained using Adam with learning rate 5e-4, $\beta_1 = 0.8$ and $\beta_2 = 0.9$ for 4000 epochs. The networks were trained on 1024 snapshots sampled from a Halton sequence of Matérn decay parameters with the following settings: $\theta = 0.1, J = 50, \nu = 0.5, l = 0.1$.\label{fig:frequency}}
    [.48\textwidth]{\includegraphics{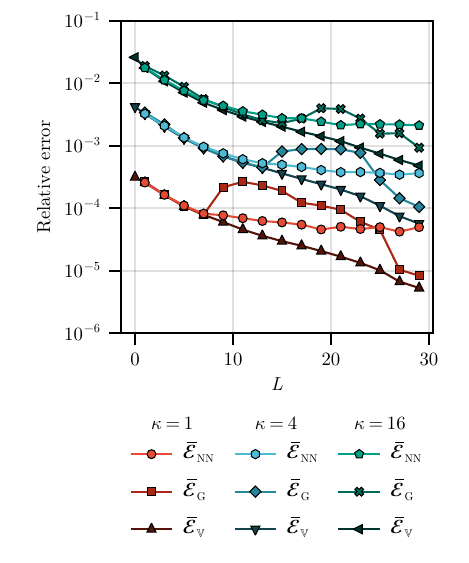}}%
    \hfill
    \captionbox{Speed up relative to the high-fidelity solver achieved with Intel(R) Xeon(R) Gold 6148 CPUs for the Galerkin-POD and POD-NN method, respectively.\label{fig:speed up}}
    [.48\textwidth]{\includegraphics{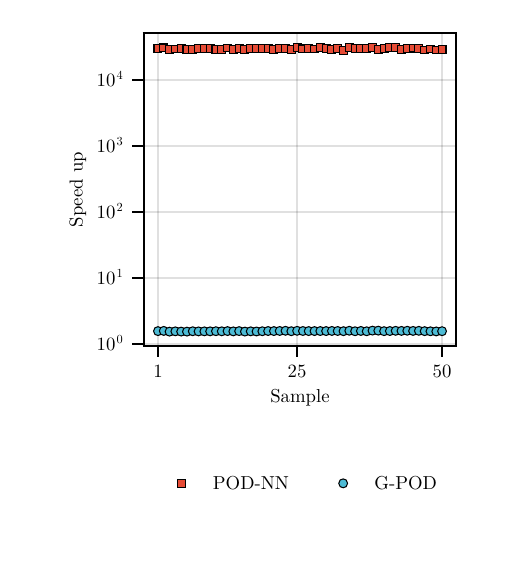}}
\end{figure}

Our experiments further indicate, that there is usually an ``optimal" number of basis functions, for which the POD-NN achieves the lowest error or does not benefit from adding more RB modes. To achieve maximum efficiency in the online phase, it may thus be desirable to truncate to a certain number of basis functions. Alternatively, adding more training data, i.e. evaluate more HF snapshots, might be necessary to train the POD-NN optimally for a higher number of basis functions: it seems that fewer snapshots are required to find a good basis than for training the NN to learn the parameter to solution map. 

\subsection{Numerical Results for the Maxwell cavity problem}
For this problem,  we follow the problem setup described in  Section \ref{eq:sec_maxwell_cavity} with the model constants set as $\omega=1, \Lambda = 1 -\imath$, $\mu=1$. The FE model uses first-order Nédélec elements with a resolution of 50 cells per side. The boundary variation is parametrized with an algebraic decay of dimension $J=10$, with scaling $\theta = 0.1$. Once again, we generate a training set on a Halton sequence of 1024 points and a test set from Latin Hypercube sampling. In Figure \ref{fig:maxwell decay}, the error of the Galerkin-POD solution perfectly follows the projection error: it decreases monotonically when more basis functions are added, confirming the success of the RB construction.
Figure \ref{fig:maxwell decay} further shows that the mean field (L=0) is probably not a sufficient predictor in this case, as the error is larger than 10\% for all algebraic decay rates. The Galerkin POD-NN also leads to satisfactory error for the strongest decay (r=3), which plateaus just below 1e-3. For the slower decay rates, the error appears to plateau for $L>2$, indicating that the NN is not able to learn the map from the parameters to the reduced coefficients for the additionally added basis functions. More data points in the parameter space are likely needed to achieve higher accuracy. In \ref{fig:maxwell freq}, we observe that the POD-NN method is also successful for higher circular frequencies of the Maxwell problem and provides vastly better accuracy than the mean prediction alone.

\begin{figure}
\centering
\captionsetup{width=.48\linewidth}
    \captionbox{Test error for the Maxwell problem with different algebraic decay rates.\label{fig:maxwell decay}}
    [.48\textwidth]{\includegraphics{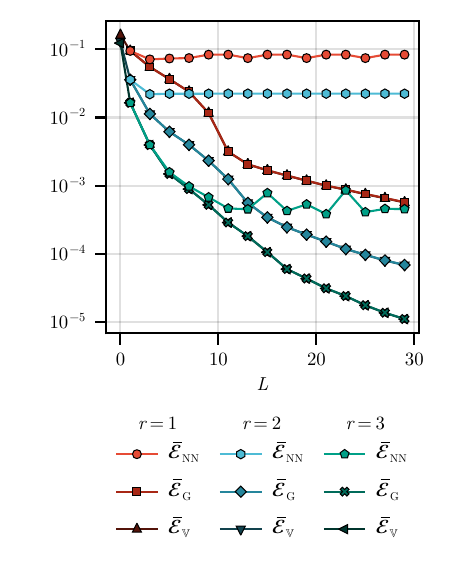}}%
    \hfill
    \captionbox{Test error for the Maxwell problem with different frequencies $\omega$ and algebraic coefficient decay with rate $r=3$.\label{fig:maxwell freq}}
    [.48\textwidth]{\includegraphics{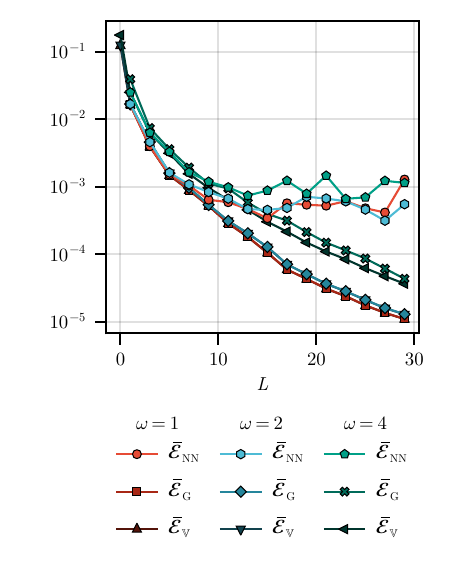}}
\end{figure}

%% file: 6_appendix_parametric_hol.tex
\section{Parametric Holomorphy of the Helmholtz Problem}\label{appendxi:proof_holomorphy}
In \cite{HSSS15}, based on a small wavenumber assumption, it is proved parametric holomorphy of the
parameter-to-solution map of the Helmholtz transmission problem. Here, we provide a complete
argument for the Helmholtz impedance problem without the aforementioned assumption.
\begin{proof}[Proof of item (ii) in Proposition \ref{prop:parametric_holomorphy}]
As it has been proved
before, the map
$\text{U} \ni \y \mapsto \boldsymbol{T}(\y) \in W^{1,\infty}(\text{D}_0; \mathbb{R}^{3 \times 3})$ is
$(\boldsymbol{b},p,\varepsilon)$-holomorphic for some $\varepsilon>0$, see e.g.~\cite{CSZ18}. 
In turn this implies that the maps
\begin{equation}\label{eq:bpes_jacobian}
	\text{U}
	\ni
	\y
	\mapsto
	\left(d \boldsymbol{T}(\y)\right)^{-\top}
	\in
	L^{\infty}\left(\text{D}_0;\mathbb{R}^{3\times 3}\right)
	\quad
	\text{and}
	\quad
	\text{U}
	\ni
	\y
	\mapsto
	J(\y)
	\in 
	L^{\infty}\left(\text{D}_0\right)
\end{equation}
are $(\boldsymbol{b},p,\varepsilon)$-holomorphic as well, which 
follows using \cite{CSZ18}. In addition, according \cite[Lemma 2.14]{dolz2023parametric}
the map 
\begin{equation}\label{eq:bpes_normal}
	\text{U} \ni \y 
	\mapsto 
	\hat{\boldsymbol{\nu}}(\y) \in L^{\infty}\left(\text{D}_0;\mathbb{R}^{3}\right)
\end{equation}
is $(\boldsymbol{b},p,\varepsilon)$-holomorphic for some $\varepsilon>0$, and as a
consequence the map 
\begin{equation}\label{eq:bpes_surface_jacobian}
	\text{U} \ni \y 
	\mapsto 
	J_{\text{S}}(\y)  \in L^{\infty}\left(\text{D}_0\right)
\end{equation}
is so as well with
the same $\boldsymbol{b} \in \ell^p(\mathbb{N})$ and $p\in (0,1)$, however possibly with a different
$\varepsilon>0$. In turn, these results imply that the map
\begin{equation}
	\text{U} \ni \y 
	\mapsto
	\hat{\mathsf{a}}
	\left(
		\cdot
		, 
		\cdot
		;
		\y
	\right) 
	\in
	\mathscr{L}
	\left(
		H^1(\text{D}_0)
		\times
		H^1(\text{D}_0)
		;
		\mathbb{C}
	\right)	
\end{equation}
is $(\boldsymbol{b},p,\varepsilon)$-holomorphic for some $\varepsilon>0$
and the same $\boldsymbol{b} \in \ell^p$ and $p\in (0,1)$,
where, for a Banach space $X$,
$\mathscr{L}\left(X \times X; \mathbb{C}\right)$ denotes the
space of continuous sesquilinear forms in $X$, which equipped
with the norm
$$
	\|\mathsf{b}\|_{\mathscr{L}\left(X \times X; \mathbb{C}\right)}
	\coloneqq
	\sup \limits_{u,v \in X \backslash \{0\}} 
	\frac{|\mathsf{b}(u,v)|}{\|u\|_X \|v\|_X},
	\quad
	\mathsf{b}
	\in
	\mathscr{L}\left(X \times X; \mathbb{C}\right),
$$
is a Banach space itself.
Furthermore, the one can also verify based on the previously stated
results that the map
\begin{equation}
	\text{U}
	\ni
	\y
	\mapsto
	\hat{\ell}(\cdot,\y)
	\in
	\left(
		H^{1}(\text{D}_0)
	\right)'
\end{equation}
is $(\boldsymbol{b},p,\varepsilon)$-holomorphic.
In addition, for each $\y \in \text{U}$ the sesquilinear form satisfies
a Garding's inequality, with in turn implies for each $\y \in \text{U}$
inf-sup conditions of the form
\begin{equation}
\begin{aligned}
	\inf_{\hat{v} \in H^1(\text{D}_0)\backslash \{0\}} 
	\sup_{\hat{w} \in H^1(\text{D}_0)\backslash \{0\}} 
	\frac{\snorm{\hat{\mathsf{a}}(\hat{v},\hat{w};\y)}}{\|u\|_X \|v\|_X}
	&
	\geq
	\alpha, \quad \text{and} \\
	\quad
	\inf_{\hat{w} \in H^1(\text{D}_0)\backslash \{0\}} 
	\sup_{\hat{v} \in H^1(\text{D}_0)\backslash \{0\}} 
	\frac{\snorm{\hat{\mathsf{a}}(\hat{v},\hat{w};\y)}}{\|u\|_X \|v\|_X}
	&
	>
	\alpha
\end{aligned}
\end{equation}
for a constant $\alpha>0$ independent of $\y \in \text{U}$.
By using a perturbation argument we may conclude
that there exists $\widetilde{\alpha}(\widetilde{\varepsilon})>0$ depending
on some $\widetilde{\varepsilon}>0$ such that for any
$(\boldsymbol{b},\widetilde{\varepsilon})$-admissible
(as in Definition \ref{def:bpe_holomorphy}),
for each $\boldsymbol{z} \in \mathcal{O}_{\boldsymbol\rho}$
we have inf-sup conditions of the form
\begin{equation}
\begin{aligned}
	\inf_{\hat{v} \in H^1(\text{D}_0)\backslash \{0\}} 
	\sup_{\hat{w} \in H^1(\text{D}_0)\backslash \{0\}} 
	\frac{\snorm{\hat{\mathsf{a}}(\hat{v},\hat{w};\z)}}{\|u\|_X \|v\|_X}
	&
	\geq
	\widetilde{\alpha}(\widetilde{\varepsilon}), \quad \text{and} \\
	\quad
	\inf_{\hat{w} \in H^1(\text{D}_0)\backslash \{0\}} 
	\sup_{\hat{v} \in H^1(\text{D}_0)\backslash \{0\}} 
	\frac{\snorm{\hat{\mathsf{a}}(\hat{v},\hat{w};\z)}}{\|u\|_X \|v\|_X}
	&
	>
	\widetilde{\alpha}(\widetilde{\varepsilon}),
\end{aligned}
\end{equation}
where for each $\boldsymbol{z} \in \mathcal{O}_{\boldsymbol\rho}$
by $\hat{\mathsf{a}}(\cdot,\cdot;\z)$ we refer to the extension of 
$\hat{\mathsf{a}}(\cdot,\cdot;\y)$ to complex-valued parametric input. 
The existence of this extension is guaranteed by the existence of equivalent extension
for the maps defined in \eqref{eq:bpes_jacobian}, \eqref{eq:bpes_normal},
and \eqref{eq:bpes_surface_jacobian}. Hence, recalling \cite[Theorem 4.1]{CCS15}
we may conclude that the map 
\begin{equation}
	\text{U}
	\ni
	\y
	\mapsto
	\hat{u}(\y)
	\in
	H^1(\text{D}_0)
\end{equation}
is $(\boldsymbol{b},p,\varepsilon)$-holomorphic 
with the same $\boldsymbol{b} \in \ell^p(\mathbb{N})$ and $p\in (0,1)$
and for some $\epsilon>0$,
where, for each $\y \in \text{U}$, $\hat{u}(\y)$ is the solution
to Problem \ref{eq: Helmholtz nominal problem}.
A similar argument holds true when we consider the discrete 
parameter-to-solution map. However, at it is customary 
for problems satisfying a Garding-type inequality, there exists
a $h_0>0$ such that for any $h<h_0$ the map
\begin{equation}
	\text{U}
	\ni
	\y
	\mapsto
	\hat{u}_h(\y)
	\in
	V_h
\end{equation}
is $(\boldsymbol{b},p,\varepsilon)$-holomorphic for some 
$\varepsilon>0$ independent of the discretization $h>0$.

\end{proof}

%% file: 7_appendix_maxwell_viz.tex
\section{Additional Figures}
\begin{figure}[ht]
    \begin{subfigure}{\textwidth}
    \centering
    \includegraphics[scale=0.2]{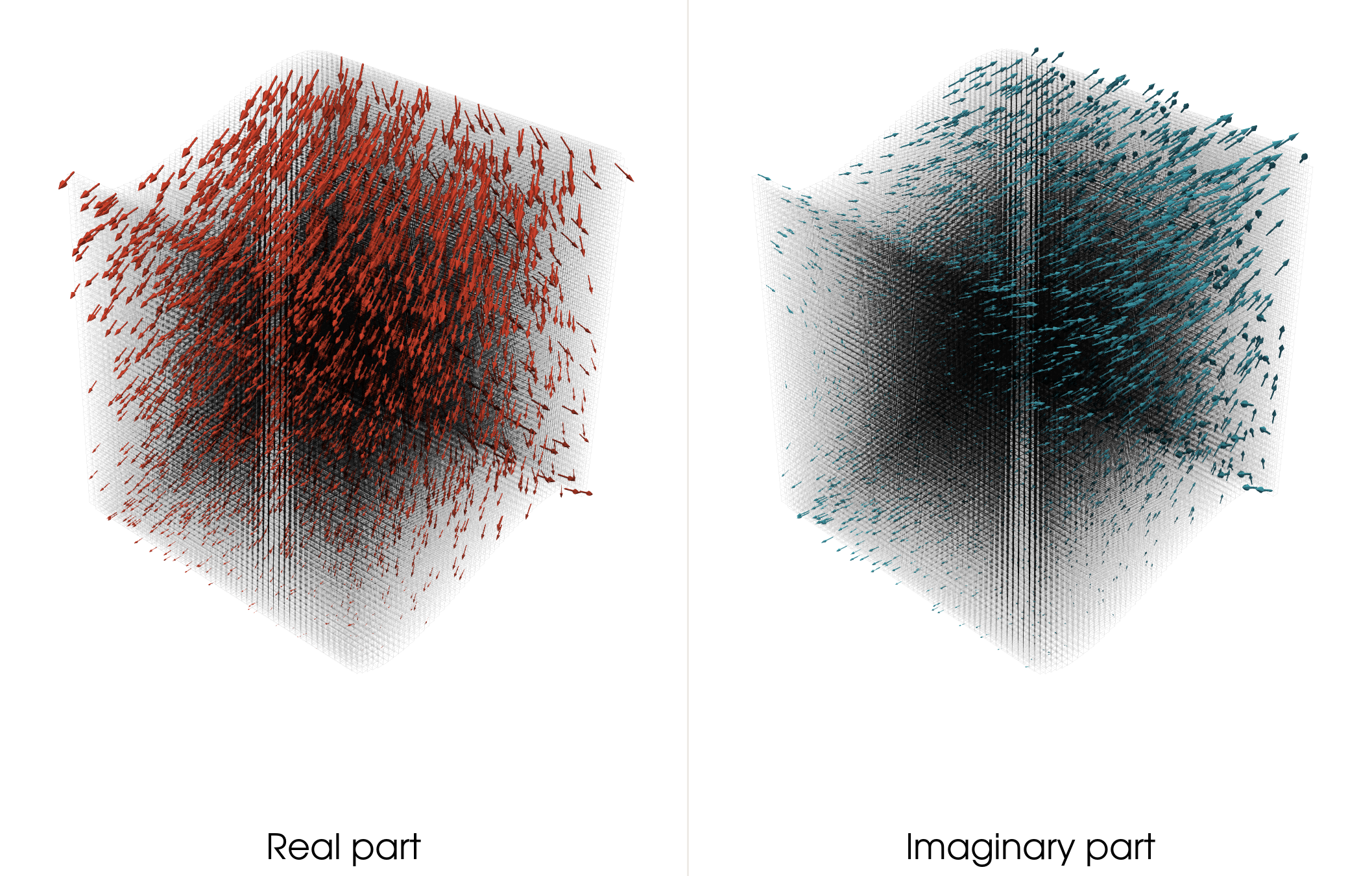}
    \end{subfigure}
    \begin{subfigure}{\textwidth}
    \centering
    \includegraphics[scale=0.2]{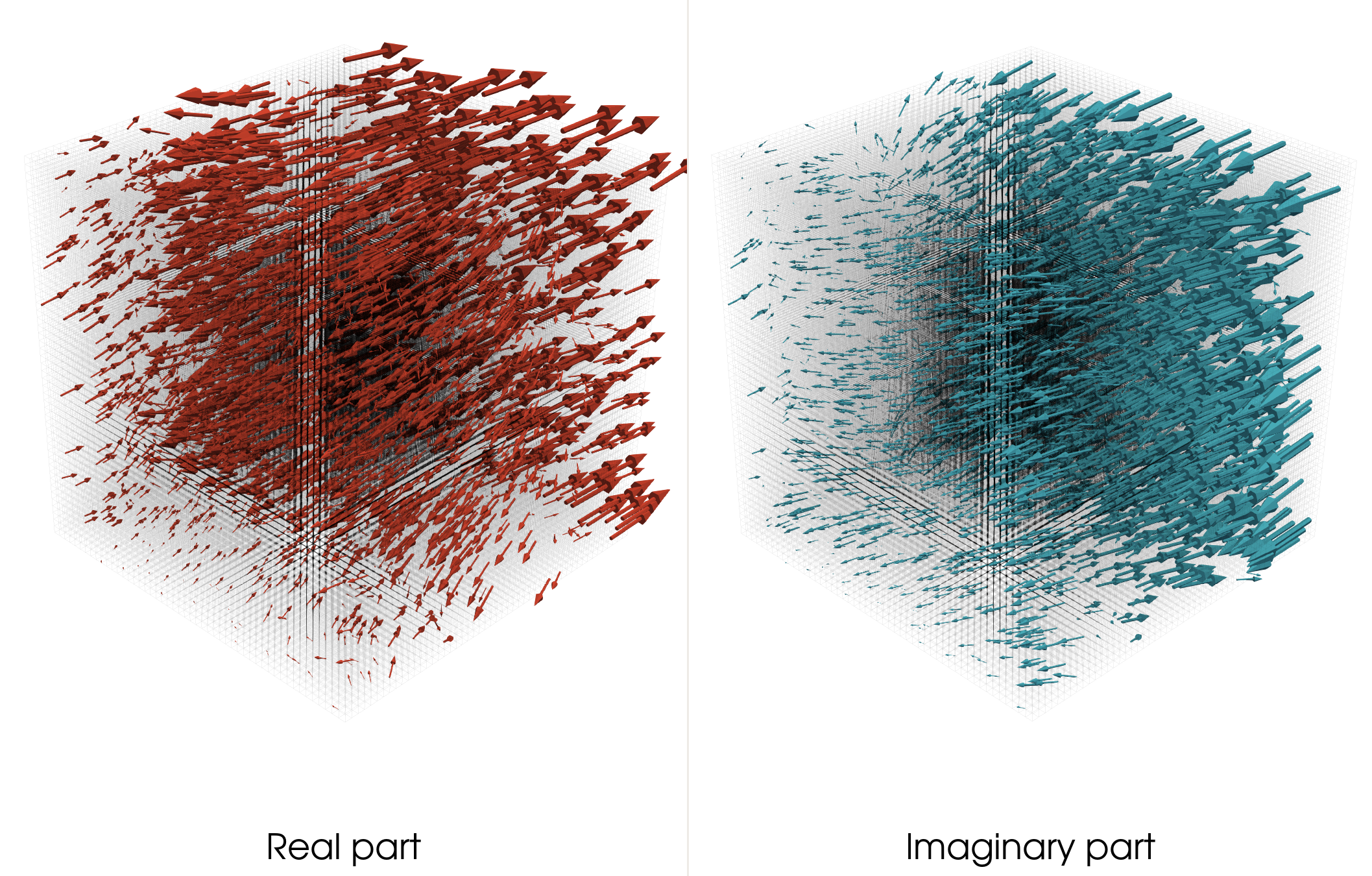}
    \end{subfigure}
    \caption{Glyph plot of the solution field (A) and the first POD mode (B) of the Maxwell problem with algebraic decay ($r=1.0$). The length of the glyphs is given by the real- and imaginary parts, respectively, where the values have been rescaled to 20\% in (B) for better visibility.}
    \label{fig:maxwell 3d visualizations}
\end{figure}

%% file: main.bbl
\begin{thebibliography}{10}

\bibitem{ABD22}
{\sc B.~Adcock, S.~Brugiapaglia, N.~Dexter, and S.~Moraga}, {\em Near-optimal
  learning of {B}anach-valued, high-dimensional functions via deep neural
  networks}, arXiv preprint arXiv:2211.12633,  (2022).

\bibitem{Aylwin2020}
{\sc R.~Aylwin, C.~Jerez-Hanckes, C.~Schwab, and J.~Zech}, {\em Domain
  uncertainty quantification in computational electromagnetics}, {SIAM}/{ASA}
  Journal on Uncertainty Quantification, 8 (2020), pp.~301--341.

\bibitem{Badia2020}
{\sc S.~Badia and F.~Verdugo}, {\em Gridap: An extensible finite element
  toolbox in {J}ulia}, Journal of Open Source Software, 5 (2020), p.~2520.

\bibitem{Bhattarai_2017}
{\sc K.~Bhattarai, S.~Silva, K.~Song, A.~Urbas, S.~J. Lee, Z.~Ku, and J.~Zhou},
  {\em Metamaterial perfect absorber analyzed by a meta-cavity model consisting
  of multilayer metasurfaces}, Scientific Reports, 7 (2017).

\bibitem{Boffi_2013}
{\sc D.~Boffi, F.~Brezzi, and M.~Fortin}, {\em Mixed Finite Element Methods and
  Applications}, Springer Berlin Heidelberg, 2013.

\bibitem{borrvall2003topology}
{\sc T.~Borrvall and J.~Petersson}, {\em Topology optimization of fluids in
  stokes flow}, International journal for numerical methods in fluids, 41
  (2003), pp.~77--107.

\bibitem{caflisch1998monte}
{\sc R.~E. Caflisch}, {\em Monte {C}arlo and quasi-monte {C}arlo methods}, Acta
  numerica, 7 (1998), pp.~1--49.

\bibitem{CNT2016}
{\sc J.~E. {Castrill{\'o}n-Cand{\'a}s}, F.~Nobile, and R.~F. Tempone}, {\em
  Analytic regularity and collocation approximation for elliptic {{PDEs}} with
  random domain deformations}, Computers \& Mathematics with Applications, 71
  (2016), pp.~1173--1197.

\bibitem{Chen_2015}
{\sc P.~Chen and C.~Schwab}, {\em Model order reduction methods in
  computational uncertainty quantification}, in Handbook of Uncertainty
  Quantification, Springer International Publishing, 2015, pp.~1--53.

\bibitem{Chen_2021}
{\sc W.~Chen, Q.~Wang, J.~S. Hesthaven, and C.~Zhang}, {\em Physics-informed
  machine learning for reduced-order modeling of nonlinear problems}, Journal
  of Computational Physics, 446 (2021), p.~110666.

\bibitem{Chen_2007}
{\sc Y.-B. Chen and Z.~Zhang}, {\em Design of tungsten complex gratings for
  thermophotovoltaic radiators}, Optics Communications, 269 (2007),
  pp.~411--417.

\bibitem{CCS15}
{\sc A.~Chkifa, A.~Cohen, and C.~Schwab}, {\em Breaking the curse of
  dimensionality in sparse polynomial approximation of parametric {PDEs}},
  Journal de Math{\'e}matiques Pures et Appliqu{\'e}es, 103 (2015),
  pp.~400--428.

\bibitem{CD16}
{\sc A.~Cohen and R.~DeVore}, {\em Kolmogorov widths under holomorphic
  mappings}, IMA Journal of Numerical Analysis, 36 (2016), pp.~1--12.

\bibitem{CSZ18}
{\sc A.~Cohen, C.~Schwab, and J.~Zech}, {\em Shape {H}olomorphy of the
  stationary {N}avier--{S}tokes equations}, SIAM Journal on Mathematical
  Analysis, 50 (2018), pp.~1720--1752.

\bibitem{DLM22}
{\sc M.~Dalla~Riva, P.~Luzzini, and P.~Musolino}, {\em Shape analyticity and
  singular perturbations for layer potential operators}, ESAIM: Mathematical
  Modelling and Numerical Analysis, 56 (2022), pp.~1889--1910.

\bibitem{DKL14}
{\sc J.~Dick, F.~Y. Kuo, Q.~T. Le~Gia, D.~Nuyens, and C.~Schwab}, {\em Higher
  order {QMC} {Petrov--Galerkin} discretization for affine parametric operator
  equations with random field inputs}, SIAM Journal on Numerical Analysis, 52
  (2014), pp.~2676--2702.

\bibitem{DLC16}
{\sc J.~Dick, Q.~T. Le~Gia, and C.~Schwab}, {\em Higher--order {Quasi--Monte
  Carlo} integration for holomorphic, parametric operator equations}, SIAM/ASA
  Journal on Uncertainty Quantification, 4 (2016), pp.~48--79.

\bibitem{dolz2023parametric}
{\sc J.~D{\"o}lz and F.~Henr{\'\i}quez}, {\em Parametric shape holomorphy of
  boundary integral operators with applications}, arXiv preprint
  arXiv:2305.19853,  (2023).

\bibitem{Ern_2017}
{\sc A.~Ern and J.-L. Guermond}, {\em Finite element quasi-interpolation and
  best approximation}, {ESAIM}: Mathematical Modelling and Numerical Analysis,
  51 (2017), pp.~1367--1385.

\bibitem{Ern_2021_I}
\leavevmode\vrule height 2pt depth -1.6pt width 23pt, {\em Finite Elements I},
  Springer International Publishing, 2021.

\bibitem{Ern_2021_II}
\leavevmode\vrule height 2pt depth -1.6pt width 23pt, {\em Finite Elements
  {II}}, Springer International Publishing, 2021.

\bibitem{GKS21}
{\sc M.~Ganesh, F.~Y. Kuo, and I.~H. Sloan}, {\em {Quasi-Monte} {Carlo} finite
  element analysis for wave propagation in heterogeneous random media},
  SIAM/ASA Journal on Uncertainty Quantification, 9 (2021), pp.~106--134.

\bibitem{guo2018reduced}
{\sc M.~Guo and J.~S. Hesthaven}, {\em Reduced order modeling for nonlinear
  structural analysis using gaussian process regression}, Computer methods in
  applied mechanics and engineering, 341 (2018), pp.~807--826.

\bibitem{HHPS2018}
{\sc A.-L. {Haji-Ali}, H.~Harbrecht, M.~Peters, and M.~Siebenmorgen}, {\em
  Novel results for the anisotropic sparse grid quadrature}, Journal of
  Complexity, 47 (2018), pp.~62--85.

\bibitem{halton1960efficiency}
{\sc J.~H. Halton}, {\em On the efficiency of certain quasi-random sequences of
  points in evaluating multi-dimensional integrals}, Numerische Mathematik, 2
  (1960), pp.~84--90.

\bibitem{HPS2016}
{\sc H.~Harbrecht, M.~Peters, and M.~Siebenmorgen}, {\em Analysis of the domain
  mapping method for elliptic diffusion problems on random domains}, Numerische
  Mathematik, 134 (2016), pp.~823--856.

\bibitem{HS2022}
{\sc H.~Harbrecht and M.~Schmidlin}, {\em Multilevel quadrature for elliptic
  problems on random domains by the coupling of {{FEM}} and {{BEM}}},
  Stochastics and Partial Differential Equations: Analysis and Computations, 10
  (2022), pp.~1619--1650.

\bibitem{henriquez2021shape}
{\sc F.~Henr\'iquez}, {\em Shape Uncertainty Quantification in Acoustic
  Scattering}, PhD thesis, ETH Zurich, 2021.

\bibitem{henriquez2024reduced}
{\sc F.~Henr{\'\i}quez and J.~Pinto}, {\em Reduced basis method for the elastic
  scattering by multiple shape-parametric open arcs in two dimensions}, arXiv
  preprint arXiv:2403.10933,  (2024).

\bibitem{HS21}
{\sc F.~Henr{\'\i}quez and C.~Schwab}, {\em Shape holomorphy of the
  {C}alder\'on projector for the {L}aplacian in $\mathbb{R}^2$}, Integral
  Equations and Operator Theory, 93 (2021), p.~43.

\bibitem{HOS22}
{\sc L.~Herrmann, J.~A. Opschoor, and C.~Schwab}, {\em Constructive deep {ReLU}
  neural network approximation}, Journal of Scientific Computing, 90 (2022),
  pp.~1--37.

\bibitem{HSZ20}
{\sc L.~Herrmann, C.~Schwab, and J.~Zech}, {\em Deep neural network expression
  of posterior expectations in {B}ayesian {PDE} inversion}, Inverse Problems,
  36 (2020), p.~125011.

\bibitem{Hesthaven2018}
{\sc J.~Hesthaven and S.~Ubbiali}, {\em Non-intrusive reduced order modeling of
  nonlinear problems using neural networks}, Journal of Computational Physics,
  363 (2018), pp.~55--78.

\bibitem{hesthaven2016certified}
{\sc J.~S. Hesthaven, G.~Rozza, B.~Stamm, et~al.}, {\em Certified reduced basis
  methods for parametrized partial differential equations}, vol.~590, Springer,
  2016.

\bibitem{HSSS15}
{\sc R.~Hiptmair, L.~Scarabosio, C.~Schillings, and C.~Schwab}, {\em Large
  deformation shape uncertainty quantification in acoustic scattering},
  Advances in Computational Mathematics, 44 (2018), pp.~1475--1518.

\bibitem{Innes2018}
{\sc M.~Innes}, {\em Flux: Elegant machine learning with julia}, Journal of
  Open Source Software, 3 (2018), p.~602.

\bibitem{JSZ16}
{\sc C.~Jerez-Hanckes, C.~Schwab, and J.~Zech}, {\em Electromagnetic wave
  scattering by random surfaces: Shape holomorphy}, Mathematical Models and
  Methods in Applied Sciences, 27 (2016), pp.~2229--2259.

\bibitem{lanthaler2023operator}
{\sc S.~Lanthaler}, {\em Operator learning with pca-net: upper and lower
  complexity bounds}, arXiv preprint arXiv:2303.16317,  (2023).

\bibitem{longo2021higher}
{\sc M.~Longo, S.~Mishra, T.~K. Rusch, and C.~Schwab}, {\em Higher-order
  quasi-monte carlo training of deep neural networks}, SIAM Journal on
  Scientific Computing, 43 (2021), pp.~A3938--A3966.

\bibitem{lye2020deep}
{\sc K.~O. Lye, S.~Mishra, and D.~Ray}, {\em Deep learning observables in
  computational fluid dynamics}, Journal of Computational Physics, 410 (2020),
  p.~109339.

\bibitem{mishra2021enhancing}
{\sc S.~Mishra and T.~K. Rusch}, {\em Enhancing accuracy of deep learning
  algorithms by training with low-discrepancy sequences}, SIAM Journal on
  Numerical Analysis, 59 (2021), pp.~1811--1834.

\bibitem{Monk}
{\sc P.~Monk}, {\em Finite Element Methods for Maxwell's Equations}, OXFORD
  UNIV PR, June 2003.

\bibitem{NTW2008}
{\sc F.~Nobile, R.~Tempone, and C.~G. Webster}, {\em An anisotropic sparse grid
  stochastic collocation method for partial differential equations with random
  input data}, SIAM Journal on Numerical Analysis, 46 (2008), pp.~2411--2442.

\bibitem{OSZ22}
{\sc J.~A. Opschoor, C.~Schwab, and J.~Zech}, {\em Exponential {ReLU} {DNN}
  expression of holomorphic maps in high dimension}, Constructive
  Approximation, 55 (2022), pp.~537--582.

\bibitem{owen1997monte}
{\sc A.~B. Owen}, {\em Monte carlo variance of scrambled net quadrature}, SIAM
  Journal on Numerical Analysis, 34 (1997), pp.~1884--1910.

\bibitem{PHJ23}
{\sc J.~Pinto, F.~Henr{\'\i}quez, and C.~Jerez-Hanckes}, {\em Shape holomorphy
  of boundary integral operators on multiple open arcs}, Journal of Fourier
  Analysis and Applications, 30 (2024), p.~14.

\bibitem{prud2002reliable}
{\sc C.~Prud'Homme, D.~V. Rovas, K.~Veroy, L.~Machiels, Y.~Maday, A.~T. Patera,
  and G.~Turinici}, {\em Reliable real-time solution of parametrized partial
  differential equations: Reduced-basis output bound methods}, J. Fluids Eng.,
  124 (2002), pp.~70--80.

\bibitem{Quarteroni_2016}
{\sc A.~Quarteroni, A.~Manzoni, and F.~Negri}, {\em Reduced Basis Methods for
  Partial Differential Equations}, Springer International Publishing, 2016.

\bibitem{rozza2014fundamentals}
{\sc G.~Rozza}, {\em Fundamentals of reduced basis method for problems governed
  by parametrized pdes and applications}, in Separated Representations and
  PGD-Based Model Reduction: Fundamentals and Applications, Springer, 2014,
  pp.~153--227.

\bibitem{SZ19}
{\sc C.~Schwab and J.~Zech}, {\em Deep learning in high dimension: {N}eural
  network expression rates for generalized polynomial chaos expansions in
  {UQ}}, Analysis and Applications, 17 (2019), pp.~19--55.

\bibitem{smith2013uncertainty}
{\sc R.~C. Smith}, {\em Uncertainty quantification: theory, implementation, and
  applications}, vol.~12, Siam, 2013.

\bibitem{sobol1967distribution}
{\sc I.~M. Sobol'}, {\em On the distribution of points in a cube and the
  approximate evaluation of integrals}, Zhurnal Vychislitel'noi Matematiki i
  Matematicheskoi Fiziki, 7 (1967), pp.~784--802.

\bibitem{SW23}
{\sc E.~A. Spence and J.~Wunsch}, {\em Wavenumber-explicit parametric
  holomorphy of helmholtz solutions in the context of uncertainty
  quantification}, SIAM/ASA Journal on Uncertainty Quantification, 11 (2023),
  pp.~567--590.

\bibitem{troltzsch2010optimal}
{\sc F.~Tr{\"o}ltzsch}, {\em Optimal control of partial differential equations:
  theory, methods, and applications}, vol.~112, American Mathematical Soc.,
  2010.

\bibitem{van2007parameter}
{\sc A.~Van~den Bos}, {\em Parameter estimation for scientists and engineers},
  John Wiley \& Sons, 2007.

\bibitem{williams2006gaussian}
{\sc C.~K. Williams and C.~E. Rasmussen}, {\em Gaussian processes for machine
  learning}, vol.~2, MIT press Cambridge, MA, 2006.

\bibitem{ZS17}
{\sc J.~Zech and C.~Schwab}, {\em Convergence rates of high dimensional
  {S}molyak quadrature}, ESAIM: Mathematical Modelling and Numerical Analysis,
  54 (2020), pp.~1259--1307.

\end{thebibliography}
